\documentclass[a4paper,11pt]{article}
\usepackage{amssymb,amsmath,amsthm}

\newtheorem{thm}{Theorem}[section]  
\newtheorem{cor}[thm]{Corollary}
\newtheorem{lem}[thm]{Lemma}
\newtheorem{prop}[thm]{Proposition}
\newtheorem{prob}[thm]{Problem}
\newtheorem{conj}[thm]{Conjecture}

\theoremstyle{definition}

\newtheorem{rem}[thm]{Remark}

\newcommand{\ve}{\varepsilon}
\newcommand{\Ex}{\mathbb{E}}
\def\Pr{\mathbb{P}}
\newcommand{\er}{\mathbb{R}}
\newcommand{\zet}{\mathbb{Z}}
\newcommand{\en}{\mathbb{N}}

\newcommand{\Med}{\mathrm{Med}}
\newcommand{\dist}{\mathrm{dist}}
\newcommand{\cala}{\mathcal{A}}

\newcommand{\Laa}{L_1}
\newcommand{\La}{L_2}
\newcommand{\Lb}{L_3}
\newcommand{\Lc}{L_4}
\newcommand{\Ld}{L_5}
\newcommand{\Le}{L_6}
\newcommand{\Lh}{L_7}
\newcommand{\Lf}{L_8}
\newcommand{\Lg}{L_9}
\newcommand{\Li}{L_{10}}
\newcommand{\Lj}{L_{11}}
\newcommand{\Lk}{L_{12}}
\begin{document}

\title{\bf On the boundedness of Bernoulli processes\thanks{Research supported by the NCN grant DEC-2012/05/B/ST1/00412}}
\author{Witold Bednorz and Rafa{\l} Lata{\l}a}
\date{}
\maketitle

\begin{abstract}
We present a positive solution to  the so-called Bernoulli Conjecture concerning the characterization
of  sample boundedness of Bernoulli processes. We also discuss some applications and related open problems.
\end{abstract}

\section{Introduction and Notation}

One of the fundamental issues of probability theory is the investigation of suprema of stochastic processes. Besides
various practical motivations it is closely related to such important theoretical problems as 
boundedness and continuity of sample paths of stochastic processes, convergence of orthogonal series, 
random series and stochastic integrals, 
estimates of norms of random vectors and random matrices, limit theorems for random vectors and empirical processes,
combinatorial matching theorems and many others.

In particular in many situations one needs to find lower and upper bounds for the quantity $\Ex\sup_{t\in T}X_t$, where 
$(X_t)_{t\in T}$ is a stochastic process. For a large class of processes (including Gaussian and Bernoulli processes)
finiteness of this quantity is equivalent to the sample boundedness, i.e. to the condition $\Pr(\sup_{t\in T}X_t<\infty)=1$. 
To avoid measurability problems one may either assume that $T$ is 
countable or define  $\Ex\sup_{t\in T}X_t:=\sup_F\Ex\sup_{t\in F}X_t$, where the supremum is taken over all finite sets 
$F\subset T$. 
The modern approach to this problem is based on chaining techniques, already present in the work of Kolmogorov and 
successfully developed over the last 40 years (see the monographs \cite{Tab1} and \cite{Tab2}). 

The most important case of centered Gaussian processes $(G_t)_{t\in T}$ is well understood. 
In this case the boundedness of 
the process is related to the geometry of the metric space $(T,d)$, where $d(t,s):=(\Ex(G_t-G_s)^2)^{1/2}$. In the landmark
paper \cite{Du}, R.~Dudley obtained an upper bound for  $g(T):=\Ex\sup_{t\in T}G_t$ in terms of entropy numbers. Dudley's
bound may be reversed for stationary processes \cite{Fe}, but not in general. In 1974
X.~Fernique \cite{Fe}  showed that for any probability measure $\mu$ on the metric space $(T,d)$,
\[
g(T)\leq L\sup_{t\in T}\int_0^{\infty}\log^{1/2}\Big(\frac{1}{\mu(B(t,x))}\Big)dx,
\]
where $L$ here and in the sequel denotes an universal constant and $B(t,x)$ is the ball in $T$ centered at $t$ with radius $x$. 
This can easily be shown to improve Dudley's estimate. 
In the seminal paper \cite{Ta_reg} M.~Talagrand showed that Fernique's
bound may be reversed, i.e. for any centered Gaussian process $G_t$ there exists a probability measure $\mu$ (called
a majorizing measure) on $T$ such that
\[
\sup_{t\in T}\int_0^{\infty}\log^{1/2}\Big(\frac{1}{\mu(B(t,x))}\Big)dx\leq Lg(T).
\] 

In general finding a majorizing measure in a concrete situation is a highly nontrivial task. In \cite{Ta_nomaj} Talagrand 
proposed a more combinatorial approach to this problem and showed that constructing a majorizing measure is equivalent 
to finding a suitable sequence of \emph{admissible} partitions of the set $T$. An increasing sequence 
$({\cal A}_n)_{n\geq 0}$ of partitions of the set $T$ is called admissible if ${\cal A}_0=\{T\}$ and 
$|{\cal A}_n|\leq N_n:=2^{2^n}$. The Fernique-Talagrand estimate may then be expressed as
\begin{equation}
\label{eq:FerTal}
\frac{1}{L}\gamma_2(T,d)\leq g(T)\leq L\gamma_2(T,d)
\end{equation}
where
\[
\gamma_{2}(T,d):=\inf\sup_{t\in T}\sum_{n=0}^{\infty}2^{n/2}\Delta(A_n(t)),
\]
and where the infimum runs over all admissible sequences  of partitions. Here  $A_n(t)$ is the unique set in ${\cal A}_n$ which
contains $t$ and $\Delta(A)$ denotes the diameter of the set $A$.

Any separable Gaussian process has a canonical Karhunen-Lo\`eve type representation $(\sum_{i=1}^{\infty}t_ig_i)_{t\in T}$,
where $g_1,g_2,\ldots$ are i.i.d. standard normal Gaussian ${\cal N}(0,1)$ r.v's  and $T$ is a subset of $\ell^2$. 
Another fundamental class of processes is obtained when in such a sum one replaces  the Gaussian r.v's $(g_i)$ by 
independent random signs. We detail this now. 

Let $I$ be a countable set and $(\ve_i)_{i\in I}$ be a Bernoulli sequence i.e.\ a sequence of
i.i.d.\ symmetric r.v's taking values $\pm 1$. For $t\in \ell^2(I)$ the series $X_t:=\sum_{i\in I}t_i\ve_i$ converges
a.s. and for $T\subset \ell^2(I)$ we may define a \emph{Bernoulli process} $(X_t)_{t\in T}$ and try to estimate
$b(T):=\Ex\sup_{t\in T}X_t$. There are two easy ways to bound $b(T)$. The first is a consequence of the uniform bound 
$|X_t|\leq \|t\|_1=\sum_{i\in I}|t_i|$, so that $b(T)\leq\sup_{t\in T}\|t\|_1$. Another is based on the domination
by the canonical Gaussian process $G_t:=\sum_{i\in I}t_ig_i$. Indeed, assuming independence
of $(g_i)$ and $(\ve_i)$,  Jensen's inequality implies
\begin{equation}
\label{eq:gaussdom}
g(T)=\Ex\sup_{t\in T}\sum_{i\in I}t_ig_i=\Ex\sum_{i\in I}t_i\ve_i|g_i|\geq \Ex\sum_{i\in I}t_i\ve_i\Ex|g_i|=
\sqrt{\frac{2}{\pi}}b(T).
\end{equation} 
Obviously also if $T\subset T_1+T_2=\{t^1+t^2\colon\ t^l\in T_l\}$ then $b(T)\leq b(T_1)+b(T_2)$, hence
\begin{align*}
b(T)
&\leq\inf\Big\{\sup_{t\in T_1}\|t\|_1+\sqrt{\frac{\pi}{2}}g(T_2)\colon\ T\subset T_1+T_2\Big\}
\\
&\leq \inf\Big\{\sup_{t\in T_1}\|t\|_1+L\gamma_2(T_2)\colon\ T\subset T_1+T_2\Big\},
\end{align*} 
where $\gamma_2(T)=\gamma_2(T,d_2)$ and $d_2$ is the $\ell^2$-distance. 
It was open for about 25 years (under the name of Bernoulli conjecture) whether the above estimate may be reversed 
(see e.g. Problem 12 in \cite{LT} or Chapter 4 in \cite{Tab1}). 
Our main result, announced in \cite{BL}, provides an affirmative answer.

\begin{thm}
\label{th:BC}
For any set $T\subset \ell^2(I)$ with $b(T)<\infty$ we may find a decomposition 
$T\subset T_{1}+T_{2}$
with $\sup_{t\in T_{1}}\sum_{i\in I}|t_{i}|\leq Lb(T)$ and $g(T_2)\leq Lb(T)$.
\end{thm}

Of course part of the difficulty is that the decomposition is neither unique nor canonical. Let us briefly describe 
some crucial ideas behind the proof, which uses a number of tools developed over the years
by Michel Talagrand. First of all we must review the proof of the lower bound of \eqref{eq:FerTal} 
in the modern approach, as in e.g. \cite{Tab1}. Every idea of this proof is used to its fullest in our approach.

As was nicely explained in \cite{Ta_sim} two fundamental facts behind this proof are Gaussian concentration and the 
Sudakov minoration principle. Gaussian concentration asserts that the fluctuations of the supremum of a Gaussian process 
are at worse like those of a single Gaussian r.v.\ with standard deviation about the  diameter of the space $(T,d)$ 
(irrelevant of the average value of this supremum). The Sudakov minoration says that the supremum of $m$ Gaussian 
r.v's with distances at least $a$ of each other is about $a\sqrt{\log m}$. 
These two principles can then be combined to obtain a ``growth condition'' as follows. 
If the space $(T,d)$ contains $m$ pieces $H_l$, which are at mutual distances at least $a$, and if each of these pieces 
is of diameter at most a small fraction of $a$, then the expected value of the supremum of the process over the whole 
index set $T$ is larger by about $a\sqrt{\log m}$ than the minimum over $l$ of the expected value of supremum of the process 
on the set $H_l$. This brings the idea to measure the ``size'' $F(A)$ of a subsets $A$ of $T$ by the expected value of 
the supremum of the process over $A$. One is then led to perform constructions in the abstract metric space $(T,d)$ using 
only the value of the ``functional'' $F(A)$ over the subsets $A$ of $T$. 
(The concept of functionals and related ``growth conditions" was introduced and developed by Talagrand \cite{Ta_AOPgc,Tab1} 
to simplify proofs and give a unified approach to various majorizing measure type results.) 
The basic ingredient to the proof  is then a ``decomposition lemma'', which is a simple consequence of the growth condition
through a ``greedy'' construction.   
Roughly speaking this decomposition lemma asserts that there exists a universal constant $r$ with the property that 
any subset $A$ of $T$ can be partitioned into at most $m$ pieces such that each piece either has the
diameter at most $\Delta(A)/r$, or else it satisfies the condition that its every subset $B$ of diameter at most 
$\Delta(A)/r^2$ satisfies $F(B)\leq F(A)- c\Delta(A) \sqrt{\log{m}}$ for some universal constant $c$. (The reader observes that the condition
on $B$ is {\it not} that its diameter is at most $\Delta(A)/r$ but the much more stringent requirement that its diameter is
at most $\Delta(A)/r^2$. It is exactly this point which makes the proof delicate.) In words, every piece is either small,
or it has the property that the value of the functional on its very small sub-pieces is quite smaller than on the whole of $A$.
The admissible sequence of partitions we look for is then obtained by a recursive use of the decomposition lemma. 
Each set $A$ belonging to ${\cal A}_n$ is partitioned in at most $N_n=2^{2^{n}}$ sets to produce the partition
${\cal A}_{n+1}$. It is not obvious, but true, that the resulting sequence of partitions has the required properties. 
(Proving this is the tricky part of the whole proof.)

When working with Bernoulli processes (and many others) the situation is more complicated than in the Gaussian case and one 
needs to use a family of distances interpolating between the $\ell^2$ and the $\ell^1$ distances.  Such  distances were  
introduced  by Talagrand in \cite{Ta_infdiv}, \cite{Ta_canon}, \cite{TaGAFA} and will be of constant use. 
An important concept in our proof is reducing the decomposition of the set $T$ to constructing a suitable  
admissible sequence of partitions. Theorem \ref{th:part} below is a refinement of previous 
results of Talagrand in the same direction, \cite{Ta_canon,TaGAFA,Tab1}.  In some sense this type of result amounts to 
organize chaining in an efficient way. Indeed in \cite{Tab2} M. Talagrand used such a result to settle the long standing 
problem of convergence of random Fourier series in a very general case. 

How, then, should one construct the required partitions?

M. Talagrand extended to  Bernoulli processes both Gaussian concentration and the Sudakov minoration in \cite{Ta_isop} and
\cite{Ta_infdiv} (see Theorems \ref{th:concBern} and \ref{th:SudBern} below). The Sudakov minoration result provides a 
lower bound on the expected value of the supremum of variables $X_{t^l}$ when the various points $t^l$ are far from each 
other in the $\ell^2$ sense, but it requires a control in the supremum norm 
of the elements $t^l$.  (The overall idea is simply that by the central limit theorem a sum $\sum_{i} \varepsilon_i t_i $ 
looks more like a Gaussian r.v.\ if all the coefficients are small.) In order to apply this minoration to increasingly larger
families, one needs to reduce the supremum norm. To do this  M. Talagrand introduced in \cite{Ta_infdiv} the fundamental idea 
of ``chopping maps''.  
These  replace the process of interest by a process where the control in the supremum norm is better, but which 
is related to the original process through an equally crucial comparison theorem (Theorem \ref{th:contr} below). 
This is essentially done by replacing each term $t_i\ve_i$ by a sum $\sum_{j} \varphi_j(t_i) \ve_{i,j}$ 
for new independent Bernoulli r.v's and certain functions $\varphi_j$, where we control uniformly $\sup |\varphi_{j}(t_i)|$, 
and where $|t_i|=\sum_j |\varphi_j(t_i)|$.  In some sense in this procedure we ``add more Bernoulli r.v's'' to the process. 

On the base of these tools Talagrand was able to prove in \cite{TaGAFA} a weaker form of Bernoulli conjecture with 
$\ell^p$-diameter bound on the set $T_1$,  $p>1$ instead of $\ell^1$-diameter. Although such a bound is not optimal, 
it was sufficient to obtain deep results about Rademacher cotype constants of operators on $C(K)$ spaces.

The main difficulty in using chopping maps optimally is that there are two $\ell^2$-distances involved, the distance associated 
to the process before it is chopped, and the possibly much smaller distance associated to the process after it is chopped.  
This makes it very difficult not to loose information during the construction. For example, if we try to mimic the construction 
in the Gaussian case, and if at a given stage of the construction we have a set $A$ with the property that on every subset of 
very small diameter the process is significantly smaller than on the whole of $A$, it is far from clear what this implies after 
applying  a chopping map since sets of small diameter for the ``smaller distance'' need not be of small diameter for the 
larger distance. Maps other than chopping maps were used in \cite{La}, where the  Bernoulli conjecture was verified for a 
very special class of subsets of $\ell^2$.  Proposition \ref{prop:prop1} below is a 
modification of the key new fact proved in that  paper. It is the cornerstone of  our paper. While Talagrand's chopping maps 
amount somehow to introduce new Bernoulli r.v's, a major new ingredient is that we find convenient at times to \emph{remove} 
some of these variables (which can only decrease the size of the process). In the situation  of Proposition \ref{prop:prop1} 
we consider a subset $J$ of $I$ and the process $X'_t=\sum_{i\in J} t_i\ve_i$; that is, we remove the Bernoulli r.v's 
which are not indexed by $J$. We then have two $\ell^2$-distances on the index set: a small one 
$\sqrt{\sum_{i\in J} (t_i-s_i)^2}$  and a large one $\sqrt{\sum_{i\in I} (t_i-s_i)^2}$. Roughly speaking the content of 
Proposition \ref{prop:prop1} is that if the index set has a small diameter with respect to the \emph{smaller} distance  
we may decompose it into not too many sets 
which either have a small diameter with respect to the \emph{larger} original distance or else have the property that the 
size of the process over the whole piece has decreased significantly when one drops the Bernoulli r.v's which are not 
indexed by $J$. The quantitative version of the result involves of course the ubiquitous term $\sqrt{\log m}$ where $m$ is 
the number of pieces permitted. 

Even after this principle has been clarified, it is still a very non-trivial technical problem to define an appropriate 
family of ``functionals'' to measure the ``size'' of the pieces of our partition.  These functionals at time ``add'' new 
Bernoulli r.v's and at time ``remove'' some. Of course the difficulty  is to find an exact balance between these two 
operations to ensure that no essential information is lost. Our functionals depend on four parameters $J, u, k, j$. 
The parameter $j\in \mathbb Z$ indicates ``how much chopping we have performed''. The other three parameters keep track of 
which Bernoulli r.v's we still use in the functional.  A new feature of this construction is that our functionals 
depend not only on which stage of the construction we are at, but also on which piece we are trying to partition. At each 
step we use a ``decomposition lemma'', which we give in Corollary \ref{cor:maindec}, somewhat similar in spirit to that 
of the Gaussian case.  Another new feature is that this lemma is not obtained only through a growth condition. To prove it 
we also apply in an essential way Proposition \ref{prop:prop1} mentioned above. In contrast with the Gaussian case, 
the decomposition lemma now produces three distinct types of pieces. Two of the types of pieces behave as in the Gaussian case. 
The new type of piece has the property that its size (as measured by the proper functional) has decreased compared to the 
set we partitioned after ignoring a suitable subset of the Bernoulli r.v's.

Our proof also uses in an essential way the technique of ``counters'' introduced by Talagrand 
to keep suitably track of the ``past'' of the construction, c.f. \cite[Chapter 5]{Tab1}.

Theorem \ref{th:BC} yields another striking characterization of boundedness for Bernoulli processes. 
For a random variable $X$ and $p>0$ we set $\|X\|_p:=(\Ex|X|^p)^{1/p}$. 

\begin{cor}
\label{cor:conv}
Suppose that $(X_t)_{t\in T}$ is a Bernoulli process with $b(T)<\infty$. Then there exist  $t^1,t^2,\ldots\in \ell^2$ 
such that
$T-T\subset \overline{\mathrm{conv}}\{t^n\colon\ n\geq 1\}$ and $\|X_{t^n}\|_{\log(n+2)}\leq Lb(T)$ for all $n\geq 1$.
\end{cor}

The converse statement easily follows from  the union bound and Chebyshev's inequality. Indeed,
suppose that $T-T\subset \overline{\mathrm{conv}}\{t^n\colon\ n\geq 1\}$ and $\|X_{t^n}\|_{\log(n+2)}\leq M$.  
Then for $u\geq 1$,
\begin{align*}
\Pr\Big(\sup_{s\in T-T}X_s\geq uM\Big)&\leq \Pr\Big(\sup_{n\geq 1}X_{t^n}\geq uM\Big)\leq
\sum_{n\geq 1}\Pr(X_{t^n}\geq u\|X_{t^n}\|_{\log(n+2)})
\\
&\leq \sum_{n\geq 1}u^{-\log(n+2)}
\end{align*}
and integration by parts easily yields $\Ex\sup_{s\in T-T}X_t\leq LM$. Moreover for any $t_0\in T$,
\[
b(T)=\Ex\sup_{t\in T}(X_t-X_{t_0})=\Ex\sup_{t\in T}(X_{t-t_0})\leq \Ex\sup_{s\in T-T}X_s\leq LM.
\]

One of the motivations to state the Bernoulli Conjecture was a question of X.\ Fernique about vector-valued random 
Fourier series (which we solve in Theorem \ref{th:Fe} below). Another interesting application of Theorem \ref{th:BC} 
is a Levy-Ottaviani type
maximal inequality for VC-classes (Theorem \ref{th:maxVC}).

To put Theorem \ref{th:BC} in a proper perspective, we will briefly explain that is it just the first step towards a much 
more ambitious program outlined in Talagrand's book \cite{Tab2}. One way to describe \eqref{eq:FerTal} in words is that 
``chaining explains the size of Gaussian processes''. The best chaining bound one can obtain for the supremum of a 
Gaussian process is of the correct order. Now, the bound $\sum_i t_i \varepsilon_i\leq \sum_i |t_i|$ on a Bernoulli process 
is of a different nature, in the sense that it makes no use of cancellation between the various terms. In some sense, 
Theorem \ref{th:BC} can be reformulated as ``chaining explains the part of boundedness which is due to cancellation''. 
That is, chaining explains the boundedness of the part $T_2$ of the process, while the boundedness of the $T_1$ part 
owes nothing to cancellation. It is argued in \cite{Tab2} that the phenomenon that  ``chaining explains the part of 
boundedness due to cancellation'' could be true in many more situations (empirical processes, infinitely divisible processes). 
Here we just briefly discuss the case of empirical processes.   

Let $(X_i)_{i\leq N}$ be i.i.d.\ r.v's with values in a measurable space $(S,{\cal S})$ and ${\cal F}$ be a class of 
measurable functions on $S$. It is a fundamental problem, strictly related to the investigation of uniform laws of large 
numbers, uniform central limit theorems and various applications in asymptotic statistics c.f. \cite{Du2,VW},
to relate the quantity
\begin{equation}
\label{eq:supemp}
\Ex\sup_{f\in {\cal F}}\sum_{i\leq N}(f(X_i)-\Ex f(X_i)) 
\end{equation}
with the geometry of the class ${\cal F}$. A first situation is when one already controls 
\[
\Ex\sup_{f\in {\cal F}}\sum_{i\leq N} |f(X_i)|,
\]
a situation where there is no cancellation. A second situation is when one can bound the quantity 
\eqref{eq:supemp} using chaining. Since then one has to use Bernstein's inequality \eqref{eq:Ber}, this requires not only 
a control  of the size of ${\cal F}$ with respect to the $\ell^2$ norm, but also with respect to the $\ell^\infty$ norm. 
Talagrand then  conjectures that the general situation is a mixture of these two cases. The precise technical statement is
given in  Conjecture \ref{conj:emp} below.

A discretized version of this problem concerning the``selector processes" based
on the i.i.d. sequence $(\delta_i)_{i\in I}$ will also be discussed in Section 9.

In a somewhat different direction, we would like to mention a very beautiful  generalization of the Bernoulli Conjecture  
formulated by S.Kwapie\'n (private communication).

\begin{prob}
Let $(F,\| \cdot \|)$ be a normed space and $(u_i)$ be a sequence of vectors in $F$ such that
the series $\sum_{i\geq 1}u_i\ve_i$ converges a.s. Does there exist a universal constant $L$ and a decomposition 
$u_i=v_i+w_i$ such that
\[
\Ex\Big\|\sum_{i\geq 1}v_ig_i\Big\|\leq L\Ex\Big\|\sum_{i\geq 1}u_i\ve_i\Big\| \quad \mbox{and}\quad 
\sup_{\eta_i=\pm 1}\Big\|\sum_{i\geq 1}w_i\eta_i\Big\|\leq L\Ex\Big\|\sum_{i\geq 1}u_i\ve_i\Big\|?
\]
\end{prob}

Theorem \ref{th:BC} shows that the answer is positive for $F=\ell^{\infty}$, in general however we may only assume that 
$F$ is a subspace of $\ell^\infty$. The difficulty here is that our proof gives very little additional information about
the decomposition given by Theorem \ref{th:BC}, in particular there is no reason for sets $T_1$ and $T_2$ to be contained 
in the linear space spanned by the index set $T$.

\medskip

The paper is organized as follows. In Section \ref{sec:est} we gather general results about Bernoulli processes. The main new
ingredient there is Proposition \ref{prop:prop1}. In Section \ref{sec:part} we show how to reduce finding a required decomposition
of the index set to constructing a suitable admissible sequence of partitions. In Section \ref{sec:chop} on the base of 
chopping maps we define functionals and in Section \ref{sec:dec} we show that they satisfy a Talagrand-type decomposition condition 
stated in Corollary \ref{cor:maindec}. 
In Section \ref{sec:constr} we inductively construct a required 
admissible sequence of partitions and conclude proofs of the main results stated above in Section \ref{sec:proofs}. 
In Section \ref{sec:appl} we present two applications of our main result and in the last Section \ref{sec:probl}
we discuss in more details the situation of  ``selector processes''.

\medskip

\noindent {\bf Acknowledgments.} We would like to thank professors Stanis{\l}aw Kwapie\'n and Michel Talagrand for 
constant encouragement to work on the problem. Upon seeing our original proof, M.  Talagrand was able 
to simplify a number of technical details, and we are grateful to him for allowing us to freely use some of his arguments.

\medskip

\noindent{\bf Notation.} By $(\ve_i)_{i}$ and $(\ve_{i,j})_{i,j}$ we  denote independent Bernoulli sequences. We use letter 
$L$ to denote positive universal constants that may change from line to line, and $L_{i}$ for positive universal constants that are the same at each occurrence. 

By $\Delta_{\ell^2(I)}(T)$ (or $\Delta_2(T)$ if the set $I$ is clear from the context) we  denote the diameter
with respect to the $\ell^2$-distance of the set $T\subset \ell^2(I)$.

\section{Estimates for Bernoulli processes}
\label{sec:est}

In the first part of this section we gather several well known estimates for suprema of Bernoulli processes and
discuss some of their consequences that play a crucial role in the proof of main result.

We start with the following simple bound on the diameter of the index set.

\begin{lem}
\label{lem:diamest}
For any $T\subset \ell^2(I)$ we have $\Delta_2(T)\leq 4b(T)$.
\end{lem}

\begin{proof}
Let $X_t:=\sum_{i}t_i\ve_i$ for $t\in T$. Take any $t,s\in T$, then
\[
b(T)\geq \Ex\max\{X_t,X_s\}=\Ex\max\{X_t-X_s,0\}=\frac{1}{2}\Ex|X_t-X_s|\geq \frac{1}{4}\|t-s\|_2.
\]
\end{proof}

Obviously by Jensen's inequality we have
\begin{equation}
\label{eq:trivialcomp}
\Ex\sup_{t\in T}\sum_{i\in J}t_i\ve_i\leq \Ex\sup_{t\in T}\sum_{i\in I}t_i\ve_i \quad \mbox{for }
J\subset I. 
\end{equation}
Much less trivial is the following Talagrand's comparison theorem for Bernoulli processes (cf. Theorem 2.1 in \cite{Ta_infdiv} 
or the proof of Theorem 4.12 in \cite{LT}).

\begin{thm}
\label{th:contr}
Suppose $\varphi_i\colon \er\rightarrow \er$, $i\in I$ are contractions (i.e. 
$|\varphi_i(x)-\varphi_i(y)|\leq |x-y|$) and $\varphi_i(0)=0$ for all $i\in I$. Then for any
$T\subset \ell^2(I)$,
\[
\Ex\sup_{t\in T}\sum_{i\in I}\varphi_i(t_i)\ve_i\leq \Ex\sup_{t\in T}\sum_{i\in I}t_i\ve_i.
\]
\end{thm} 

\begin{rem}
Since 
\[
\Ex\sup_{t\in T}\sum_{i\in I}\varphi_i(t_i)\ve_i=
\Ex\sup_{t\in T}\sum_{i\in I}(\varphi_i(t_i)-\varphi_i(0))\ve_i
\]
we may replace the assumption that $\varphi_i(0)=0$ with $(\varphi_i(0))\in \ell^2(I)$ (which for contractions is equivalent to
$(\varphi_i(t_i))\in \ell^2(I)$ for some/all $t\in \ell^2(I)$).
\end{rem}

A typical application of Theorem \ref{th:contr} is the following.

\begin{cor}
\label{cor:contr}
Suppose that $(f_{i,j})$ and $(g_i)$ are functions on $\er$ such that for all $i\in I$, $x,y\in\er$,
\[
\sum_{j\in J}|f_{i,j}(x)-f_{i,j}(y)|\leq |g_i(x)-g_i(y)|. 
\]
Let $T$ be a set such that $(g_i(t_i))\in \ell^2(I)$ and $(f_{i,j}(t_i))\in \ell^2(I\times J)$ for all $t\in T$. Then
\[
\Ex\sup_{t\in T}\sum_{i\in I,j\in J}f_{i,j}(t_i)\ve_{i,j}\leq 
\Ex\sup_{t\in T}\sum_{i\in I}g_i(t_i)\ve_i.
\]
\end{cor}

\begin{proof}
Without loss of generality we may assume that the sequences $(\ve_{i,j})$ and $(\ve_i)$ are independent. It is enough to observe
that
\[
\Ex\sup_{t\in T}\sum_{i\in I,j\in J}f_{i,j}(t_i)\ve_{i,j}=
\Ex\sup_{t\in T}\sum_{i\in I}\Big(\sum_{j\in J}f_{i,j}(t_i)\ve_{i,j}\Big)\ve_i
\]
and that for any values of $\ve_{i,j}\in \{\pm 1\}$ and $x,y\in \er$,
\[
\Big|\sum_{j\in J}f_{i,j}(x)\ve_{i,j}-\sum_{j\in J}f_{i,j}(y)\ve_{i,j}\Big|\leq |g_i(x)-g_i(y)|.
\]
The assertion follows by applying conditionally Theorem \ref{th:contr}.
\end{proof}

Next we state the concentration property of Bernoulli processes (cf. \cite{Ta_isop} or
\cite[Corollary 4.10]{Le}).

\begin{thm}
\label{th:concBern}
Let $(a_t)_{t\in T}$ be a sequence of real numbers indexed by a set $T\subset \ell^2(I)$ and
$S:=\sup_{t\in T}(a_{t}+\sum_{i\in I}t_{i}\ve_{i})$ be such that
$|S|<\infty$ a.s. Then
\[
\Pr(|S- \Med(S)|\geq u)\leq 4\exp\Big(-\frac{u^{2}}{16\sigma^{2}}\Big)
\quad \mbox{for } u>0,
\]
where $\sigma:=\sup_{t\in T}\|t\|_{2}$. In particular $\Ex|S|<\infty$,
$|\Ex S-\Med(S)|\leq L\sigma$ and
\[
\Pr(|S- \Ex(S)|\geq u)\leq 2\exp\Big(-\frac{u^{2}}{\Laa\sigma^{2}}\Big)
\quad \mbox{ for } u>0.
\]
\end{thm}

Theorem \ref{th:concBern} easily implies the following fact \cite[Corollary 1]{La}.

\begin{prop}
\label{prop:sup_m}
Let $(Y_{t}^{k})_{t\in T}$, $1\leq k\leq m$ be i.i.d. Bernoulli processes
and $\sigma:=\sup_{t\in T}\|Y_{t}^{1}\|_{2}$. Then for any
process $(Z_{t})_{t\in T}$ independent of $(Y_{t}^{k}\colon t\in T, k\leq m)$
 we have
\[
\Ex \max_{1\leq k\leq m}\sup_{t\in T}(Z_t+Y_{t}^{k})
\leq \Ex \sup_{t\in T}(Z_t+Y_{t}^{1})+ \La\sigma\sqrt{\log m}.
\]
\end{prop}

Another important property of Bernoulli processes is a Sudakov-type minoration formulated and proved by Talagrand
(cf. \cite{Ta_infdiv} or \cite[Theorem 4.2.4]{Tab1}).

\begin{thm}
\label{th:SudBern}
Suppose that vectors $t_{1},\ldots,t_{m}\in \ell^{2}(I)$ and numbers $a,b>0$ satisfy
\begin{equation}
\label{eq:assSud1}
\forall_{l\neq l'}\ \|t_{l}-t_{l'}\|_{2}\geq a\quad
\mbox{ and }\quad \forall_{l}\ \|t_{l}\|_{\infty}\leq b.
\end{equation}
Then
\[
\Ex\sup_{l\leq m}\sum_{i\in I}t_{l,i}\ve_{i}\geq
\frac{1}{\Lb}\min\Big\{a\sqrt{\log m},\frac{a^{2}}{b}\Big\}.
\]
\end{thm}

Our next proposition combines concentration and minoration properties for Bernoulli processes
\cite[Proposition 4.2.2]{Tab1}. It exactly parallels the Gaussian case. 

\begin{prop}
\label{prop:min+conc}
Consider vectors  $t_{1},\ldots,t_{m}\in \ell^{2}(I)$ and numbers $a,b>0$
such that \eqref{eq:assSud1} holds. Then for any $\sigma>0$ and any sets
$H_{l}\subset B_{\ell^{2}(I)}(t_{l},\sigma)$,
\[
b\Big(\bigcup_{l\leq m}H_{l}\Big)\geq
\frac{1}{\Lc}\min\Big\{a\sqrt{\log m},\frac{a^{2}}{b}\Big\}
-\Ld\sigma\sqrt{\log m}+\min_{l\leq m} b(H_{l}).
\]
\end{prop}

Proposition \ref{prop:min+conc} together with a simple greedy algorithm  yields the following 
decomposition result for Bernoulli processes. This again parallels the Gaussian case. 

\begin{cor}
\label{cor:greedy}
Suppose that $\|t\|_{\infty}\leq b$ for all $t\in T$ and $b\sqrt{\log m}\leq \sigma$. 
Then there exists sets $C_1,\ldots,C_{m-1}\subset T$ such that $\Delta_{\ell^2(I)}(C_i)\leq \Le\sigma$
and for each nonempty set $D\subset T\setminus\bigcup_{k\leq m-1}C_k$ with  $\Delta_{\ell^2(I)}(D)\leq \sigma$,
\[
b(D)\leq b(T)-\sigma\sqrt{\log m}.
\]
\end{cor}

\begin{proof}
Let $\Le=\max\{2,2\Lc(\Ld+2)\}$ and $a=\frac{1}{2}\Le\sigma$. Then
\[
\min\Big\{a\sqrt{\log m},\frac{a^{2}}{b}\Big\}=a\sqrt{\log m}\geq \Lc(\Ld+2)\sigma\sqrt{\log m}.
\]
If $T\subset \bigcup_{i\leq m-1}B(t_i,a)$ 
for some $t_1,\ldots,t_{m-1}\in T$  there is nothing to prove, otherwise 
we choose inductively vectors $t_1,t_2,\ldots,t_{m-1}$. To this end we set $T_1:=T$ and 
$T_k:=T\setminus\bigcup_{l<k}B(t_l,a)$ for $k>1$ and choose $t_k\in T_k$ in such a way
that
\[
b(T_k\cap B(t_k,\sigma))\geq \sup_{t\in T_k}b(T_k\cap B(t,\sigma))-\sigma\sqrt{\log m}.
\]
Let $C_k:=T\cap B(t_k,a)$ for $k\leq m-1$. Then obviously $\Delta_{\ell^2(I)}(C_k)\leq \Le\sigma$.
Take any $D\subset T_m=T\setminus\bigcup_{k< m}C_k$ with  $\Delta_{\ell^2(I)}(D)\leq \sigma$
and choose any $t_{m}\in D$ so that $D\subset B(t_{m},\sigma)\cap T_m$. By  construction the condition
\eqref{eq:assSud1} holds. Let $H_l:=B(t_l,\sigma)\cap T_l$, for $l<m$ and $H_m:=D$. Then by the choice of $t_l$
it follows that 
\[
\min_{1\leq l\leq m}b(H_l)\geq b(D)-\sigma\sqrt{\log m}.
\]
So by Proposition \ref{prop:min+conc}
\begin{align*}
b(T)&\geq b\Big(\bigcup_{l\leq m}H_{l}\Big)\geq
\frac{1}{\Lc}\min\Big\{a\sqrt{\log m},\frac{a^{2}}{b}\Big\}+b(D)-(\Ld+1)\sigma\sqrt{\log m}
\\
&\geq b(D)+\sigma\sqrt{\log m}.
\end{align*}
\end{proof}

The last result of this section is a modification of Proposition 1 from \cite{La}, which will be crucial in the proof of
the main decomposition result, Corollary \ref{cor:maindec}.
Before we state it
let us introduce a bit of notation. For $\emptyset\neq J\subset I$, $t\in \ell^{2}(I)$, $T\subset \ell^{2}(I)$ we define
$t_{J}:=(t_{i})_{i\in J}\in \ell^{2}(J)$,
\[
  b_{J}(T):= \Ex \sup_{t\in T}\sum_{i\in J}\ve_{i}t_{i},
\]
\[
  d_{J}(t,s):=\|t_{J}-s_{J}\|_{2},\ \ t,s\in \ell^{2}(I)
\]
and
\[
  B_{J}(t,a):=\{s\in \ell^{2}(I)\colon d_{J}(s,t)\leq a\},\  a\geq 0.
\]

\begin{prop}
\label{prop:prop1}
Consider a positive integer $m$, numbers $b,c,\sigma>0$ and $\lambda\geq 1$ that satisfy
$b\sqrt{\log m}\leq \lambda\sigma$ and $T\subset \ell^{2}(I)$  such that
\begin{equation}
\label{eq:assum1}
\forall_{t,s\in T}\  d_{J}(t,s)\leq c,\ \
\|t-s\|_{\infty}\leq b.
\end{equation}
Then there exist
$t_{1},\ldots,t_{m}\in T$ such that either
$T\subset \bigcup_{l\leq m}B_{I}(t_{l},\sigma)$ or
\begin{equation}
\label{eq:remest}
b_{J}\Big(T\setminus \bigcup_{l\leq m}B_{I}(t_{l},\sigma)\Big)\leq
b_{I}(T)-\Big(\frac{1}{4\lambda\Lb}\sigma-\Lh c\Big)\sqrt{\log m}.
\end{equation}
\end{prop}

Observe that we use in Proposition \ref{prop:prop1} two distances $d_J$ and $d_I$. What is fundamental here is that
we assume that the diameter of the set $T$ is small only with respect to the smaller distance $d_J$ and we show that it
may be covered by a certain number of balls with respect to the larger distance $d_I$ and a remaining set with a small
value of $b_J$.

\begin{proof}
If $T\subset \bigcup_{l\leq m}B_{I}(t_{l},\sigma)$ for some $t_1,\ldots,t_m\in T$
or $m=1$ there is nothing to prove, so we will assume
that this is not the case. 
We may also choose the universal constant $\Lh$ in such a way that $\Lb\Lh\geq 1$, so it is enough to
consider the case $\sigma\geq 2c$ (since otherwise $\frac{1}{4\lambda\Lb}\sigma-\Lh c<0$). 

Since $b_{J}(T)=b_{J}(T-t)$ for any $t\in \ell^{2}(I)$, we may and will
assume that $0\in T$, so that
\[
\|t_J\|_2\leq c,\ \ \|t\|_{\infty}\leq b\leq \frac{\lambda\sigma}{\sqrt{\log m}} \quad \mbox{ for } t\in T.
\]
We need to show that 
\begin{equation}
\label{eq:toshow}
\alpha< b_{I}(T)-\Big(\frac{1}{4\lambda\Lb}\sigma-\Lh c\Big)\sqrt{\log m},
\end{equation}
where
\[
\alpha:=\inf_{t_{1},\ldots,t_{m}\in T}
b_{J}\Big(T\setminus \bigcup_{l\leq m}B_{I}(t_{l},\sigma)\Big).
\]

Let $\ve_{i}^{(k)}$, $i\in J,\ k=1,\ldots,m$  be independent
Bernoulli r.v's, independent of $(\ve_{i})_{i\in I}$. Let
\[
  Y_{t}^{(k)}:=\sum_{i\in J}t_{i}\ve_{i}^{(k)},\ \
  Z_{t}:=\sum_{i\in I\setminus J}t_{i}\ve_{i}.
\]
Then for any $k$,
\[
b(T)=\Ex\sup_{t\in T}(Z_t+Y_t^{(k)}),
\]
and therefore Proposition \ref{prop:sup_m} yields
\begin{equation}
\label{eq:est1}
\Ex\max_{1\leq k\leq m}\sup_{t\in T}(Z_t+Y_t^{(k)})\leq b(T)+\La c\sqrt{\log m}.
\end{equation}
We set $T_1=T$ and define a random point $t_1\in T_1$ that depends only on $(\ve^{(1)}_i)_{i\in J}$ such that
\[
Y_{t_1}^{(1)}>\sup_{t\in T_1}Y_{t}^{(1)}-c\sqrt{\log m}.
\]
We continue this construction and inductively define random points $t_k\in T$, 
$k\leq m$ that depend only on $(\ve^{(l)}_i)_{l\leq k, i\in J}$. If $t_1,\ldots,t_{k-1}$ are already defined
we set
\[
T_k:=T\setminus \bigcup_{l\leq k-1}B_I(t_l,\sigma)
\]
and we choose a random point $t_k\in T_{k}$ such that
\[
Y_{t_k}^{(k)}> \sup_{t\in T_{k}}Y_{t}^{(k)}-c\sqrt{\log m}.
\]
The process $(Y_t^{(k)})$ is independent of the set $T_{k}$ and for $k\leq m$,
\[
Y_{t_k}^{(k)}+c\sqrt{\log m}> \sup_{t\in T_k}Y_{t}^{(k)}\quad \mbox{and}\quad
\Ex\sup_{t\in T_k}Y_{t}^{(k)}\geq \alpha.
\]
We have
\begin{align}
\notag
\Ex&\max_{1\leq k\leq m}\sup_{t\in T}(Z_t+Y_t^{(k)})
\geq 
\Ex\Big(\max_{1\leq k\leq m}Z_{t_k}+\min_{1\leq k\leq m}Y_{t_k}^{(k)}\Big)
\\
\notag
&\geq
\Ex\max_{1\leq k\leq m}Z_{t_k}+\alpha-c\sqrt{\log m}+
\Ex\min_{1\leq k\leq m}\Big(\sup_{t\in T_k}Y_t^{(k)}-\alpha\Big)
\\
\label{eq:est2}
&\geq \Ex\max_{1\leq k\leq m}Z_{t_k}+\alpha-c\sqrt{\log m}+
\Ex\min_{1\leq k\leq m}\Big(\sup_{t\in T_k}Y_t^{(k)}-\Ex\sup_{t\in T_k}Y_t^{(k)}\Big).
\end{align}
Observe that  for $1\leq l<k\leq m$,
\[
d_{I\setminus J}(t_k,t_l)\geq d_I(t_k,t_l)-d_J(t_k,t_l)\geq \sigma -c\geq \frac{1}{2}\sigma,  
\]
and hence Theorem \ref{th:SudBern} with $a=\sigma/2$ (and using independence  of $Z_t$ and of the random points $(t_k)$) gives
\begin{equation}
\label{eq:est3}
\Ex\max_{1\leq k\leq m}Z_{t_k}\geq \frac{1}{4\lambda \Lb}\sigma\sqrt{\log m}.
\end{equation}

Since $(Y_t^{(k)})$ is independent on the set $T_k$, Theorem \ref{th:concBern} gives that for $u>0$,
\[
\Pr\Big(\sup_{t\in T_k}Y_{t}^{(k)}-\Ex\sup_{t\in T_k}Y_t^{(k)}\leq -u\Big)
\leq 2\exp\Big(-\frac{u^{2}}{\Laa c^{2}}\Big),
\]
so that 
\[
\Pr\Big(\min_{k\leq m}\Big(\sup_{t\in T_k}Y_t^{(k)}-\Ex\sup_{t\in T_k}Y_t^{(k)}\Big)\leq -u\Big)
\leq \min\Big\{1,2m\exp\Big(-\frac{u^{2}}{\Laa c^{2}}\Big)\Big\},
\]
and integration by parts yields
\begin{equation}
\label{eq:est4}
\Ex\min_{k\leq m}\Big(\sup_{t\in T_k}Y_t^{(k)}-\Ex\sup_{t\in T_k}Y_t^{(k)}\Big)
\geq -Lc\sqrt{\log m}.
\end{equation}
Estimates \eqref{eq:est1}-\eqref{eq:est4} imply \eqref{eq:toshow} and complete the proof.
\end{proof}

\section{Partitions}
\label{sec:part}

Following Talagrand we connect in this section decompositions of the set $T$ with 
suitable sequences of its partitions. We recall that an increasing sequence 
$({\cal A}_n)_{n\geq 0}$ of partitions of $T$ is called 
\emph{admissible} if ${\cal A}_0=\{T\}$ and $|{\cal A}_n|\leq N_n:=2^{2^n}$. For such partitions and 
$t\in T$ by $A_n(t)$ we denote by ${\cal A}_n$ the unique set which contains $t$. To each set $A\in {\cal A}_n$
we will associate a point $\pi_n(A)$ and an integer $j_n(A)$. To simplify the notation we set 
$j_n(t):=j_n(A_n(t))$ and $\pi_n(t):=\pi_n(A_n(t))$. The main new feature in the next theorem is the introduction 
of the sets $I_n(A)$.

\begin{thm}
\label{th:part}
Suppose that $M>0$, $r\geq 2$, $({\cal A}_n)_{n\geq 0}$ is an admissible sequence of partitions of $T\subset \ell^2(I)$,  
and for each  $A\in {\cal A}_n$ there exists an integer $j_n(A)$ and a point $\pi_n(A)\in T$ satisfying the 
following assumptions:\\
i) $\|t-s\|_2\leq \sqrt{M}r^{-j_0(T)}$ for $t,s\in T$,\\
ii) if $n\geq 1$, ${\cal A}_n\ni A\subset A'\in{\cal A}_{n-1}$ then either\\ 
a) $j_n(A)=j_{n-1}(A')$ and $\pi_n(A)=\pi_{n-1}(A')$\\
or\\
b) $j_{n}(A)>j_{n-1}(A')$, $\pi_n(A)\in A'$ and  
\[
\sum_{i\in I_n(A)}\min\{(t_i-\pi_n(A)_i)^2,r^{-2j_n(A)}\}\leq M2^nr^{-2j_n(A)}\mbox{ for all }t\in A,
\]
where for any $t\in A$,
\[
I_n(A)=I_n(t):=\big\{i\in I\colon\ |\pi_{k+1}(t)_i-\pi_k(t)_i|\leq r^{-j_k(t)} 
\mbox{ for } 0\leq k\leq n-1\big\}.
\]
Then there exist sets $T_1,T_2$ such that $T\subset T_1+T_2$ and
\begin{equation}
\label{eq:estdec}
\sup_{t^1\in T_1}\|t^1\|_1\leq LM\sup_{t\in T}\sum_{n=0}^{\infty}2^nr^{-j_n(t)}\ \ \mbox{ and }
\ \ \gamma_2(T_2)\leq L\sqrt{M}\sup_{t\in T}\sum_{n=0}^{\infty}2^nr^{-j_n(t)}.
\end{equation}
\end{thm}

\noindent
{\bf Remark.} Note that if $t,s\in A\in {\cal A}_n$ then for $0\leq k\leq n$, $A_k(t)=A_k(s)$ and as a consequence
$j_k(t)=j_k(s)$, $\pi_k(t)=\pi_k(s)$ and $I_n(t)=I_n(s)$. 
Therefore the definition of $I_n(A)$ does not depend on the choice of $t\in A$.  

\begin{proof}
Obviously we may assume that $\sup_{t\in T}\sum_{n\geq 0}2^nr^{-j_n(t)}<\infty$, which in particular implies that 
$\lim_{n\rightarrow\infty}j_n(t)=\infty$. Define 
\[
m(t,i):=\inf\big\{n\geq 0\colon\ |\pi_{n+1}(t)_i-\pi_{n}(t)_i|>r^{-j_n(t)}\big\}, \quad t\in T, i\in I,
\]
so that  $I_n(t)=\{i\colon\ m(t,i)\geq n\}$ for $n\geq 0$.

Observe that 
\begin{equation}
\label{eq:estpi1}
|\pi_{n+1}(t)_i-\pi_n(t)_i|\leq r^{-j_n(t)}I_{\{j_{n+1}(t)>j_n(t)\}}
\quad \mbox{for } 0\leq n<m(t,i). 
\end{equation}
Since $j_n(t)$ is nondecreasing sequence of integers, for $i$ such that $m(t,i)=\infty$ the limit 
$\pi_{\infty}(t)_i:=\lim_{n\rightarrow\infty}\pi_n(t)_i$ exists. Therefore we may define $\pi(t)$ by the formula
\[
\pi(t)_i:=\pi_{m(t,i)}(t)_i,\quad t\in T,i\in I.
\]
We set
\[
T_1:=\{t-\pi(t)\colon\ t\in T\}\quad \mbox{and} \quad T_2:=\{\pi(t)\colon\ t\in T\},
\] 
so that obviously $T\subset T_1+T_2$.

To estimate $\|t-\pi(t)\|_1$ we define
\[
\tau(t,i):=\inf\big\{n\geq 0\colon\ |\pi_n(t)_i-t_i|>\frac{1}{2}r^{-j_n(t)}\big\}, \quad t\in T, i\in I
\]
and
\[
J_n(t):=\{i\in I\colon\ \tau(t,i)=n\}.
\]
Observe that $\tau(t,i)\leq m(t,i)+1$ and if $\tau(t,i)=\infty$ then $\pi(t)_i=\pi_{\infty}(t)_i=t_i$. Therefore we have
\[
\|t-\pi(t)\|_1=\sum_{n=0}^{\infty}\sum_{i\in J_n(t)}|t_i-\pi_{m(t,i)}(t)_i|.
\]

From  \eqref{eq:estpi1} we get
\[
|\pi_0(t)_i-\pi_{m(t,i)}(t)_i|\leq \sum_{n=0}^{m(t,i)-1}|\pi_{n+1}(t)_i-\pi_n(t)_i|\leq
\sum_{j=j_0(t)}^\infty r^{-j}\leq 2r^{-j_0(t)},
\]
and moreover for $i\in J_0(t)$, it holds that  $|t_i-\pi_0(t)_i|\geq \frac{1}{2}r^{-j_0(t)}$. Thus
\begin{align*}
\sum_{i\in J_0(t)}|t_i-\pi_{m(t,i)}(t)_i|
&\leq 5\sum_{i\in J_0(t)}|t_i-\pi_0(t)_i|\leq
10r^{j_0(t)}\sum_{i\in I}|t_i-\pi_0(t)_i|^2
\\
&\leq 10Mr^{-j_0(t)},
\end{align*}
where the last estimate follows by the assumption i).

If $i\in J_n(t)$, $n\geq 1$ then $m(t,i)\geq n-1$ and
\begin{align*}
|t_i-\pi_{m(t,i)}(t)_i|
&\leq |t_i-\pi_{n-1}(t)_i|+\sum_{k=n-1}^{m(t,i)-1}|\pi_{k+1}(t)_i-\pi_{k}(t)_i|
\\
&\leq \frac{1}{2}r^{-j_{n-1}(t)}+\sum_{k=n-1}^{\infty}r^{-j_{k}(t)}I_{\{j_{k+1}(t)>j_{k}(t)\}}
\\
&\leq \frac{1}{2}r^{-j_{n-1}(t)}+\sum_{l=j_{n-1}(t)}^{\infty}r^{-l}\leq 3r^{-j_{n-1}(t)}.
\end{align*}
Hence
\[
\|t-\pi(t)\|_1\leq 10Mr^{-j_0(t)}+3\sum_{n=1}^{\infty}r^{-j_{n-1}(t)}|J_n(t)|.
\]

To estimate $|J_n(t)|$ for $n\geq 1$ we may assume that $j_{n}(t)>j_{n-1}(t)$, 
since otherwise assumption ii)a) yields $\pi_{n}(t)=\pi_{n-1}(t)$ and $|J_n(t)|=0$.
For $i\in J_n(t)$ we have either $i\in I_n(t)$ or $m(t,i)=n-1$.
Since $|\pi_n(t)_i-t_i|> \frac{1}{2}r^{-j_{n}(t)}$ for $i\in J_n(t)$ we get by the assumption ii)b)
\[
\frac{1}{4}r^{-2j_{n}(t)}|J_n(t)\cap I_n(t)|\leq 
\sum_{i\in I_{n}(t)}\min\{|t_i-\pi_n(t)_i|^2,r^{-2j_{n}(t)}\}\leq M2^nr^{-2j_{n}(t)}.
\]
If $m(t,i)=n-1$ then $|\pi_n(t)-\pi_{n-1}(t)|> r^{-j_{n-1}(t)}$. Let $n':=\inf\{k\leq n-1\colon j_k(t)=j_{n-1}(t)\}$.
Then, since $\pi_{n}(t)\in A_{n-1}(t)\subset A_{n'}(t)$, $j_{n-1}(t)=j_{n'}(t)>j_{n'-1}(t)$ and $\pi_{n-1}(t)=\pi_{n'}(t)$, 
the assumption ii)b) used this time for $n'$  yields
\begin{align*}
r^{-2j_{n-1}(t)}|\{i\colon\ m(t,i)=n-1\}|
&\leq  \sum_{i\in I_{n'}(t)}\min\{|\pi_n(t)_i-\pi_{n-1}(t)_i|^2,r^{-2j_{n-1}(t)}\}
\\
&\leq M2^{n-1}r^{-2j_{n-1}(t)}.
\end{align*}
Thus
\[
|J_n(t)|\leq |J_n(t)\cap I_n(t)|+|\{i\colon\ m(t,i)=n-1\}|\leq 9M2^{n-1}
\]
and
\[
\|t-\pi(t)\|_1\leq 10Mr^{-j_0(t)}+27M\sum_{n=1}^{\infty}2^{n-1}r^{-j_{n-1}(t)}
\leq 37M\sup_{t\in T}\sum_{n=0}^{\infty}2^{n}r^{-j_{n}(t)}.
\]

To bound   $\gamma_2(T_2)$ we will construct sets $U_n\subset \ell^2(I)$ 
such that $|U_0|=1$, $|U_n|\leq N_{n}$ for $n\geq 1$ and use \cite[Theorem 1.3.5]{Tab1}
to get
\begin{equation}
\label{eq:estT2}
\gamma_2(T_2)\leq L\sup_{t\in T}\sum_{n=0}^{\infty}2^{n/2}\dist(\pi(t),U_n).
\end{equation}

To this end we define
\[
U_n:=\{\pi_{m(t,i)\wedge n}(t)\colon\ t\in T\},
\]
where $\pi_{m(t,i)\wedge n}(t)=(\pi_{m(t,i)\wedge n}(t)_i)_{i\in I}$.   
Observe that for $s\in A_{n}(t)$, $\pi_k(s)=\pi_{k}(t)$ for $k\leq n$ and 
$\{i\colon\ m(t,i)\geq n\}=\{i\colon\ m(s,i)\geq n\}$ so that $m(t,i)\wedge n=m(s,i)\wedge n$. Hence 
$|U_n|\leq |\cala_{n}|\leq N_{n}$ for $n\geq 1$ and $U_0=\{\pi_0(T)\}$.

To estimate $\dist(\pi(t),U_n)$, first notice that
\[
\dist(\pi(t),U_n)\leq \|\pi(t)-\pi_{m(t,i)\wedge n}(t)\|_2
\leq \sum_{l=n}^{\infty}\|(\pi_{l+1}(t)-\pi_{l}(t))1_{\{m(t,i)\geq l+1\}}\|_2.
\]
The condition $m(t,i)\geq l+1$ implies $|\pi_{l+1}(t)_i-\pi_{l}(t)_i|\leq r^{-j_l(t)}$.
If $j_{l+1}(t)=j_l(t)$ then $\pi_{l+1}(t)=\pi_{l}(t)$, otherwise $\pi_{l+1}(t)\in A_l(t)$ and by the assumption ii)b)
\begin{align*}
\|(\pi_{l+1}(t)-\pi_{l}(t))1_{\{m(t,i)\geq l+1\}}\|_2^2
&\leq \sum_{i\in I_{l+1}(t)}\min\{|\pi_{l+1}(t)_i-\pi_{l}(t)_i|^2,r^{-2j_l(t)}\}
\\
&\leq M2^lr^{-2j_l(t)}.
\end{align*}
Therefore 
\[
\dist(\pi(t),U_n)\leq \sum_{l=n}^{\infty}\sqrt{M}2^{l/2}r^{-j_l(t)}
\]
and
\[
\sum_{n=0}^{\infty}2^{n/2}\dist(\pi(t),U_n)\leq \sqrt{M}\sum_{l=0}^{\infty}2^{l/2}r^{-j_l(t)}\sum_{n=0}^l2^{n/2}
\leq L\sqrt{M}\sum_{l=0}^{\infty}2^{l}r^{-j_l(t)}.
\]
Hence the estimate for $\gamma_2(T_2)$ follows by \eqref{eq:estT2}.
\end{proof}

\section{Chopping maps}
\label{sec:chop}

In this section on the base of the so-called chopping maps we define functionals that will play a key role
in the proof of Theorem \ref{th:BC}. Chopping maps were introduced by Talagrand in \cite{Ta_infdiv},
he used them to prove a weak form of the Bernoulli Conjecture (\cite{TaGAFA} and \cite[Section 4.1]{Tab1}). 

For $u<v$ we define the non-increasing function $\varphi_{u,v}$ by the formula
\[
\varphi_{u,v}(x):=\min\{v,\max\{x,u\}\}-\min\{v,\max\{0,u\}\}.
\]
In other words $\varphi_{u,v}$ is the unique continuous function, which  is constant on half lines $(-\infty,u]$ and 
$[v,\infty)$, has slope 1 on the interval $[u,v]$ and takes value $0$ at $0$. Observe that 
$|\varphi_{u,v}(x)|\leq v-u$, $|\varphi_{u,v}(x)-\varphi_{u,v}(y)|\leq |x-y|$ and
\begin{equation}
\label{eq:chopp1}
\varphi_{u_0,u_k}(x)=\sum_{l=1}^k\varphi_{u_{l-1},u_l}(x)\quad \mbox{ for }
u_0<u_1<\ldots<u_k.
\end{equation}

\begin{lem}
For any $u_0<u_1<\ldots<u_k$ and $x,y\in \er$ we have
\begin{equation}
\label{eq:chopp2}
\sum_{l=1}^k|\varphi_{u_{l-1},u_l}(x)-\varphi_{u_{l-1},u_l}(y)|=
|\varphi_{u_0,u_k}(x)-\varphi_{u_0,u_k}(y)|\leq |x-y|.
\end{equation}
In particular
\begin{equation}
\label{eq:chopp3}
\sum_{l=1}^k|\varphi_{u_{l-1},u_l}(x)|\leq |x|\quad \mbox{ and }\quad
\sum_{l=1}^k\varphi_{u_{l-1},u_l}(x)^2\leq x^2.
\end{equation}
\end{lem}

\begin{proof}
W.l.o.g. we may assume that $x>y$. Then $\varphi_{u,v}(x)\geq \varphi_{u,v}(y)$ for any $u,v$ and
\eqref{eq:chopp2} follows by \eqref{eq:chopp1}. The ``In particular'' part easily follows taking $y=0$.
\end{proof}

Let $G_i=\{u_{i,0}<u_{i,1}<\ldots<u_{i,k_i}\}$, $i\in I$ be a family of finite subsets of $\er$ and
${\cal G}=(G_i)_{i \in I}$. For $t\in \ell^2(I)$ we define Bernoulli processes
\[
X_{t}(G_i,i):=\sum_{l=1}^{k_i}\varphi_{u_{i,l-1},u_{i,l}}(t_i)\ve_{i,l}
\] 
and
\[
X_{t}({\cal G}):=\sum_{i\in I}X_{t}(G_i,i)=\sum_{i\in I}\sum_{l=1}^{k_i}\varphi_{u_{i,l-1},u_{i,l}}(t_i)\ve_{i,l}.
\]
Note that for $t\in \ell^2(I)$ by \eqref{eq:chopp3} we get 
\[
\sum_{i\in I}\sum_{l=1}^{k_i}|\varphi_{u_{i,l-1},u_{i,l}}(t_i)|^2\leq \sum_{i\in I} t_i^2<\infty
\]
and $X_t({\cal G})$ is well defined. We also consider the canonical distance $d_{{\cal G}}$ associated to the process
$X_t({\cal G})$ given by
\[
d_{{\cal G}}(s,t)^2:=\Ex|X_t({\cal G})-X_s({\cal G})|^2=
\sum_{i\in I}\sum_{l=1}^{k_i}|\varphi_{u_{i,l-1},u_{i,l}}(t_i)-\varphi_{u_{i,l-1},u_{i,l}}(s_i)|^2.
\]

\begin{prop}
\label{prop:functcalG}
i) For any family of finite sets ${\cal G}=(G_i)_{i\in I}$ and $T\subset \ell^2(I)$ we have
\[
\Ex\sup_{t\in T}X_t({\cal G})\leq b(T)=\Ex\sup_{t\in T}\sum_{i\in I}t_i\ve_i.
\]
ii) If ${\cal G}=(G_i)_{i\in I}$ and ${\cal G'}=(G_i')_{i\in I}$ are two families of finite subsets of $\er$ such that
for all $i\in I$,
\begin{equation}
\label{eq:GGprime}
G_i\subset G_i',\ \max_i G_i=\max_i G_i' \mbox{ and } \min_i G_i=\min_i G_i'
\end{equation}
then for any $T\subset \ell^2(I)$,
\[
\Ex\sup_{t\in T}X_t({\cal G}')\leq \Ex\sup_{t\in T}X_t({\cal G}).
\]
\end{prop}

\begin{proof}
Part i) follows easily by Corollary \ref{cor:contr} and \eqref{eq:chopp2}.

To show part ii) let $G_i=\{u_{i,0}<u_{i,1}<\ldots<u_{i,k_i}\}$
and $[u_{i,l-1},u_{i,l}]\cap G_i'=\{s_{i,l,0}<s_{i,l,1}<\ldots<s_{i,l,k_{i,l}}\}$. Then
\[
\Ex\sup_{t\in T}X_t({\cal G}')=
\Ex\sup_{t\in T}\sum_{i\in I}\sum_{l=1}^{k_i}\sum_{j=1}^{k_{i,l}}\varphi_{s_{i,l,j-1},s_{i,l,j}}(t_i)\ve_{i,l,j}
\]
and the assertion follows by  Corollary \ref{cor:contr} and \eqref{eq:chopp2}.
\end{proof}

Inequality \eqref{eq:chopp3} yields
\begin{equation}
\label{eq:distest}
d_{{\cal G}}(s,t)\leq \|s-t\|_2\quad \mbox{for }s,t\in \ell^2(I).
\end{equation}

The next proposition shows how to compare $d_{{\cal G}}$ with $d_{{\cal G'}}$.

\begin{prop}
\label{prop:distcalG}
Let ${\cal G}=(G_i)_{i\in I}$ and ${\cal G'}=(G_i')_{i\in I}$ be two families of finite subsets of $\er$
such that $G_i\subset G_i'$ and $G_i=\{u_{i,0}<u_{i,1}<\ldots<u_{i,k_i}\}$ for all $i\in I$.\\
i) If $\max_i G_i=\max_i G_i'$  and  $\min_i G_i=\min_i G_i'$ then $d_{{\cal G'}}\leq d_{{\cal G}}$.\\
ii) If $|G_i'\cap (u_{i,l-1},u_{i,l}]|\leq q$ for all $i\in I$, $1\leq l\leq k_i$ then $d_{{\cal G}}\leq \sqrt{q}d_{{\cal G'}}$.
\end{prop}

\begin{proof}
Part i) follows by \eqref{eq:chopp2} and the inequality $\sum_{l}|a_l|^2\leq (\sum_{l}|a_l|)^2$. To show
ii) we also use  \eqref{eq:chopp2} and the bound $(\sum_{l=1}^k|a_l|)^2\leq k \sum_{l=1}^k|a_l|^2$.
\end{proof}

We are now ready to define functionals and related distances.
Let $r\geq 4$ be an integer to be chosen later.
For $x\in \er$ and $k\in \zet$ we set
\[
G(x,k):=\{pr^{-k}\colon\  p\in\zet\}\cap [x-4r^{-k},x+4r^{-k}).
\]
In other words if $p_k(x)=\lceil r^kx\rceil \in \zet$, i.e. $(p_k(x)-1)r^{-k}<x\leq p_k(x)r^{-k}$ then
\[
G(x,k)=\{pr^{-k}\colon p_k(x)-4\leq p\leq p_k(x)+3\}.
\]
For an integer $j\geq k$ we set
\begin{align*}
G(x,k,j)&
:=\{pr^{-j}\colon\ (p_k(x)-4)r^{-k}\leq pr^{-j}\leq (p_k(x)+3)r^{-k}\}
\\
&= \{pr^{-j}\colon\ w_{k,j}(x)\leq p\leq v_{k,j}\},  
\end{align*}
where $w_{k,j}(x):=(p_k(x)-4)r^{j-k}$ and $v_{k,j}(x):=(p_k(x)+3)r^{j-k}$.
Then $G(x,k,k)=G(x,k)$ and 
\begin{align}
\notag
j'\geq j\geq k\ \Rightarrow\ & G(x,k,j)\subset G(x,k,j'),\ \min G(x,k,j)=\min G(x,k,j')
\\
\label{eq:compGx}
 &\mbox{ and } \max G(x,k,j)=\max G(x,k,j').
\end{align}

For $u\in \ell^2(I)$, integers $j\geq k$  and $J\subset I$ we define the process $X_t(J,u,k,j)$ by
\[
X_t(J,u,k,j):=X_{t}((G(u_i,k,j))_{i\in J})=
\sum_{i\in J}\sum_{p=w_{k,j}(u_i)+1}^{v_{k,j}(u_i)}\varphi_{(p-1)r^{-j},pr^{-j}}(t_i)\ve_{i,p}.
\]
For $T\subset \ell^2(I)$ we set
\[
F(T,J,u,k,j):=\Ex\sup_{t\in T}X_t(J,u,k,j).
\]

Increasing the parameter $j$ corresponds to the ``adding" new Bernoulli r.v's, while increasing
the parameter $k$ results in ``removing" some of Bernoulli r.v's from the process $X_t(J,u,k,j)$.

Let us denote by $d(J,u,k,j)$ the canonical distance associated to the process $(X_t(J,u,k,j))$, i.e.
\[
d(J,u,k,j)(t,s):=\Big(\Ex (X_t(J,u,k,j)-X_s(J,u,k,j))^2\Big)^{1/2}
\]
and let $\Delta(T,J,u,k,j)$ denote the diameter of the set $T\subset \ell^2(I)$ with respect to $d(J,u,k,j)$.

Proposition \ref{prop:functcalG}i) and \eqref{eq:distest} easily yield the following.

\begin{prop}
For any $J\subset I$, $u\in \ell^2(I)$, integers $j\geq k$ and $T\subset \ell^2(I)$ we have
\[
F(T,J,u,k,j)\leq b(T) 
\] 
and
\[
\Delta(T,J,u,k,j)\leq \Delta_{\ell^2(I)}(T).
\]
\end{prop}

We also have the following comparison of distinct functionals and related distances.

\begin{prop}
\label{prop:monfunct}
If $J'\subset J\subset I$, integers $j\geq k$ and $j'\geq k'$ satisfy $j'\geq j$ and $k'\geq k$ then
for any $u\in \ell^2(I)$ and $T\subset \ell^2(I)$ we have
\[
F(T,J',u,k',j')\leq F(T,J,u,k,j)
\]
and
\[
\Delta(T,J',u,k',j')\leq \Delta(T,J,u,k,j).
\]
\end{prop}

\begin{proof}
The monotonicity of $F(T,J,u,k,j)$ with respect to the set $J$ and the variable $k$ easily follows by the definition of 
$X_t(J,u,k,j)$ and $\eqref{eq:trivialcomp}$. The monotonicity with respect to $j$ is a consequence of Proposition
\ref{prop:functcalG} ii) and \eqref{eq:compGx}.

Monotonicity of distances $d(T,J,u,k,j)$ with respect to $J$ and $k$ is quite obvious, and with respect to $j$
follows by Proposition \ref{prop:distcalG}.
\end{proof}

We conclude this section with a lemma that gives lower bound for the constructed distances.

\begin{lem}
\label{lem:estd}
For $s,t,u\in \ell^2(I)$, $J\subset I$ and $j\geq k$,
\[
d(J,u,k,j)(t,s)^2\geq \frac{1}{2}\sum_{i\in J}\min\{|s_i-t_i|^2,r^{-2j}\}I_{\{|s_i-u_i|\leq 2r^{-k}\}}. 
\]
\end{lem}

\begin{proof}
It is easy to reduce to the case when $|s_i-u_i|\leq 2r^{-k}$ and $|s_i-t_i|\leq r^{-j}$ for all $i\in J$.
Then for any $i\in J$, $\min G(u_i,k,j)\leq s_i\leq t_i\leq \max G(u_i,k,j)$ and for at most two integers $p$,
$\varphi_{(p-1)r^{-j},pr^{-j}}(t_i)\neq \varphi_{(p-1)r^{-j},pr^{-j}}(s_i)$. 
The estimate follows by \eqref{eq:chopp1}, since $(a+b)^2\leq 2a^2+2b^2$. 
\end{proof}

\section{Decomposition Lemmas}
\label{sec:dec}

In this section we derive several decomposition results for our functionals $F(T,J,u,k,j)$.
First two propositions are based on results of Section \ref{sec:est}. We combine them to get Corollary \ref{cor:maindec}
on which we will base our inductive construction of suitable partitions.

The first proposition immediately follows from Corollary \ref{cor:greedy}.

\begin{prop}
\label{prop:funct1}
Let $T\subset \ell^2(I)$, $u\in \ell^2(I)$, $J\subset I$ and $j\geq k$.
If $r^{-j}\sqrt{\log m}\leq \sigma$   then there exist
sets $C_1,\ldots, C_{m-1}\subset T$ such that 
\[
\Delta(C_l,J,u,k,j)\leq \Le\sigma \quad 1\leq l\leq m-1
\]
and for any $\emptyset\neq D\subset T\setminus\bigcup_{l<m}C_l$
with $\Delta(D,J,u,k,j)\leq \sigma$, it holds
\[
F(D,J,u,k,j)\leq F(T,J,u,k,j)-\sigma\sqrt{\log m}.
\]  
\end{prop}

The next result is crucial.

\begin{prop}
\label{prop:funct2}
Let $u,u'\in \ell^2(I)$, $J\subset I$, $j\geq k$ and $J'\subset J$ be such that 
$|u_i-u'_i|\leq 2r^{-k}$
for all $i\in J'$. Let $T$ be a subset of $\ell^2(I)$ with $\Delta(T,J,u,k,j+2)\leq c$.
If $r^{-j-1}\sqrt{\log m}\leq \sigma$ and $\Lf c\leq \sigma$ 
then there exist sets $A_1,\ldots,A_m\subset T$ such that
\[
\Delta(A_l,J,u,k,j+1)\leq \sigma\quad \mbox{for } 1\leq l\leq m
\]
and either $T\subset \bigcup_{l\leq m}A_l$ or
\begin{equation}
\label{eq:funct2r}
F\Big(T\setminus\bigcup_{l=1}^m A_l,J',u',j+2,j+2\Big)\leq 
F(T,J,u,k,j+1)- \frac{1}{\Lg}\sigma\sqrt{\log m}.
\end{equation}
\end{prop}  

\begin{proof}
Let ${\cal G}=(G_i)_{i\in J}$, ${\cal G'}=(G_i')_{i\in J}$, where
$$
G_i=G(u_i,k,j+1),\ i\in J$$
and
\[
G_i'=\left\{\begin{array}{ll}
G_i & \mbox{ for }i\in J\setminus J',\\
G_i\cup G(u_i',j+2,j+2)  & \mbox{ for }i\in J'.
\end{array}\right.
\]
Since $r\geq 4$ and $j\geq k$ we have
\[
G(u_i',j+2,j+2)\subset [u_i'-4r^{-j-2},u_i'+4r^{-j-2})\subset (u_i'-r^{-k},u_i'+r^{-k}).
\]
Moreover $|u_i-u'_i|\leq 2r^{-k}$ for $i\in J'$, and  therefore the sets $G_i$ and $G_i'$ satisfy the condition
\eqref{eq:GGprime} and Proposition \ref{prop:functcalG}ii) yields
\[
\Ex\sup_{t\in T}X_t({\cal G'})\leq \Ex\sup_{t\in T}X_t({\cal G})=F(T,J,u,k,j+1).
\] 
Since $|G(u_i',j+2,j+2)|=8$, Proposition \ref{prop:distcalG}ii) with $q=9$ yields $d_{{\cal G}}\leq 3d_{{\cal G'}}$.

For $i\in J'$ we have $|u_i-u_i'|\leq 2r^{-k}$, so that
\[
|pr^{-j-2}-u_i'|\leq 4r^{-j-2}\ \Rightarrow |pr^{-j-2}-u_i|\leq 2r^{-k}+4r^{-j-2}\leq 3r^{-k}
\]
and therefore $G(u_i',j+2,j+2)\subset G(u_i,k,j+2)$. Thus
\[
\Delta(T,J',u',j+2,j+2)\leq \Delta(T,J,u,k,j+2)\leq c.
\]

We apply Proposition \ref{prop:prop1} with $b=r^{-j-1},\lambda=6$ and $\sigma^*,I^*,J^*,T^*$ instead of $I,J$ and $T$, where
$\sigma^*:=\sigma/6$,
\[
I^*:=\{(i,u)\colon\ i\in J,\ u\in G_i'\setminus \{\min G_i\}\},
\]
\[
J^*:=\{(i,u)\colon\ i\in J',\ u\in  G(u_i',j+2,j+2)\setminus\{\min G(u_i',j+2,j+2)\}\}
\]
and for $A\subset T$,
\[
A^*:=\{(\varphi_{u-,u}(t_i))_{(i,u)}\colon\ t\in A,\ (i,u)\in I^*\},  
\]
where for $(i,u)\in I^*$, $u-$ denotes the largest element of $G_i'$ smaller than $u$.
Observe that with the notation of Proposition \ref{prop:prop1} we have for $A\subset T$
\[
b_{I^*}(A^*)=\Ex\sup_{t\in A}X_t({\cal G'})\quad \mbox{and}\quad b_{J^*}(A^*)=F(A,J',u',j+2,j+2).
\]
It is not hard to check that all the  assumptions of the proposition are satisfied. Hence there exist sets $A_1,\ldots,A_m\subset T$
such that $A_l^*\subset B_{I^*}(t^*_l,\sigma^*)$ for some $t_l^*\in T^*$ and 
\begin{align*}
F\Big(T\setminus\bigcup_{l=1}^m A_l,J',u',j+2,&j+2\Big)
\leq \Ex\sup_{t\in T}X_t({\cal G'})-\Big(\frac{1}{144\Lb}\sigma-\Lh c\Big)\sqrt{\log m}
\\
&\leq F(T,J,u,k,j+1)-\Big(\frac{1}{144\Lb}\sigma-\Lh c\Big)\sqrt{\log m}.
\end{align*}
Hence condition \eqref{eq:funct2r} holds if we take $\Lf=288\Lb\Lh$ and $\Lg=288\Lb$.
We conclude by observing that the condition $A_l^*\subset B_{I^*}(t_l^*,\sigma^*)$ implies that for $s,t\in A_l$, we have
$d_{{\cal G}}(s,t)\leq 3d_{{\cal G'}}(s,t)\leq 6\sigma^*=\sigma$, and hence
$\Delta(A_l,J,u,k,j+1)\leq \sigma$, $1\leq l\leq m$.  
\end{proof}

We finish this section with a key corollary which states that our functionals satisfy a Talagrand-type decomposition condition  
Namely each set may be decomposed into pieces of three types. Pieces of type (C3) have small diameters and pieces of type (C1)
have small value of a functional on subsets with sufficiently small diameters, in both cases we do not change values of parameters $k,J$ and $u$. Pieces satisfying conditions (C2) are of different type -- they have both small diameters and 
small value of functionals, however we increase the parameter $k$ and allow changes in parameters $u$ and $J$.

\begin{cor}
\label{cor:maindec}
There exists a positive integer $r_0$ with the following property. 
Consider $T\subset \ell^2(I)$, $J\subset I$, $u\in \ell^2(I)$, $u'\in T$, $c\geq 0$ and
integers $j\geq k$, $n\geq 1$, $r\geq r_0$ and set
\[
J':=\{i\in J\colon\ |u_i-u'_{i}|\leq 2r^{-k}\}.
\]
Then we can find $p\leq N_n$ and a partition 
$(A_l)_{l\leq p}$ of $T$ such that each set $A_l$ satisfies one of the following properties:\\
\begin{align}
\notag
\mbox{for any }D\subset A_l \mbox{ with }\Delta(D,J,u,k,j+2)\leq \frac{1}{\Li}2^{n/2}r^{-j-1}
\\
\tag{C1}
F(D,J,u,k,j+2)\leq F(T,J,u,k,j+2)-\frac{1}{\Lj}2^nr^{-j-1}
\end{align}
or
\begin{equation}
\tag{C2a}
\Delta(A_l,J',u',j+2,j+2)\leq \Delta(A_l,J,u,k,j+2)\leq 2^{n/2}r^{-j-1},
\end{equation}
\begin{align}
\notag
F(A_l,J',u',j+2,j+2)
&\leq F(T,J,u,k,j+1)-\frac{1}{\Lk}2^nr^{-j-1}
\\
\tag{C2b}
&\leq F(T,J,u,k,j)-\frac{1}{\Lk}2^nr^{-j-1}
\end{align}
or
\begin{equation}
\tag{C3}
\Delta(A_l,J,u,k,j+1)\leq 2^{n/2}r^{-j-1}. 
\end{equation}
\end{cor}

\begin{proof}
Let $m:=\sqrt{N_{n}}$ so that $\sqrt{\log m}=2^{(n-1)/2}\sqrt{\log 2}$. Without loss of generality we may also
assume $\Lf\geq 1$ (where $\Lf$ is the absolute constant given by Proposition \ref{prop:funct2}).

We first apply Proposition \ref{prop:funct1} with $j+2$  and $\sigma=\frac{1}{\Le\Lf}2^{n/2}r^{-j-1}$.
Observe that $r^{-j-2}\sqrt{\log m}\leq r^{-j-2}2^{(n-1)/2}\leq \sigma$ if $r_0\geq \Le\Lf$. This way we obtain the 
decomposition $T=\bigcup_{l\leq m-1}C_l\cup A_1$, where  $\Delta(C_l,J,u,k,j+2)\leq c:=\frac{1}{\Lf}2^{n/2}r^{-j-1}$
and $A_1$ satisfies the condition (C1) with $\Li:=\Le\Lf$, $\Lj:=(2/\log(2))^{1/2}\Le\Lf$.

Now for $l\leq m-1$ we apply  Proposition  \ref{prop:funct2} with $T=C_l$, $\sigma=2^{n/2}r^{-j-1}$ and we decompose 
$C_l$ into at most $m+1$  sets that satisfy either (C2b) with $\Lk:=(2/\log(2))^{1/2}\Lg$ or (C3).
Since $G(u_i',j+2,j+2)\subset G(u_i,k,j+2)$ for $i\in J'$ and $\Lf\geq 1$ we get 
$\Delta(A_l,J',u',j+2,j+2)\leq \Delta(A_l,J,u,k,j+2)\leq c\leq  2^{n/2}r^{-j-1}$ and (C2a) follows.

This way we decompose the set $T$ into at most $1+(m-1)(m+1)=N_n$ sets $A_l$ satisfying one of the conditions 
(C1)-(C3). 
\end{proof}

\section{Partition construction}
\label{sec:constr}

To prove Theorem \ref{th:BC} with the use of Theorem \ref{th:part} we need to construct a suitable admissible 
sequence of partitions $({\cal A}_n)_{n\geq 0}$ of the index set $T$. In this section we present such a  construction.

We  use the following notation. For $A\in {\cal A}_n$, $n\geq 1$ by $A'$ we will denote the unique
set in ${\cal A}_{n-1}$ such that $A\subset A'$. For $t\in T$ and $n\geq 0$, $A_n(t)$ is the unique
element of ${\cal A}_n$ which contains $t$. Moreover if to each set $A\in {\cal A}_n$  is assigned
a certain quantity (which may be a point, a number or a set) $\alpha_n(A)$, then to shorten the notation
we write $\alpha_n(t)$ for $\alpha_n(A_n(t))$.

The following simple lemma will be very useful. It was proven in \cite{Tab2}, we 
rewrite its proof for the sake of completeness. 

\begin{lem}[{\cite[Lemma 2.6.3]{Tab2}}]
\label{lem:smart}
Let $\alpha>1$ and $(a_n)_{n\geq 0}$ be a sequence of positive numbers such that $\sup_{n}a_n<\infty$. Define
\[
V:=\{m\geq 0\colon\ a_n<a_m\alpha^{|n-m|}\mbox{ for all }n\geq 0,n\neq m\}.
\] 
Then
\[
\sum_{n\geq 0}a_n\leq\frac{2\alpha}{\alpha-1}\sum_{m\in V}a_m.
\]
\end{lem}

\begin{proof}
We define a partial order on $\en$ by $n\prec m$ if and only if $a_m\geq a_n\alpha^{|n-m|}$. Then  $V$ is just the set
of maximal elements of $\prec$, i.e. if $m\in V$, $m\prec m'$ then $m'=m$. Moreover, since $a_n$ is bounded
there cannot exist an infinite  sequence of integers increasing with respect to $\prec$. Therefore
for each $n\in \en$ there exists $m\in V$ such that $n\prec m$. Thus
\[
\sum_{n\geq 0}a_n\leq \sum_{m\in V}a_m\sum_{n\geq 0}\alpha^{-|n-m|}\leq
\frac{2\alpha}{\alpha-1}\sum_{m\in V}a_m.
\]
\end{proof}

We are now ready to describe the partition construction. It is based on the iterative application of Corollary
\ref{cor:maindec}. Unfortunately we will need to control several parameters. The integers $k_n\leq j_n$, the 
points $u_n\in T$ and the sets $J_n\subset I$ are related to the functionals studied in the previous
sections. The parameter $p_n=0$ means that we will use Corollary \ref{cor:maindec} to decompose
the set and $p_n>0$ means that we will wait $2\kappa-p_n$ steps before doing it. 

Let us first summarize the main dependencies between these quantities. The first condition 
gives initial values of parameters
\begin{equation}
\tag{P1}
p_0(T)=0,\ j_0(T)=k_0(T)=j_0,\ J_0(T)=I.
\end{equation}
The next requirement is a mild regularity condition (in all conditions below we assume that
$A\in {\cal A}_n$ for some $n\geq 1$) 
\begin{equation}
\tag{P2}
j_{n-1}(A')\leq j_n(A)\leq j_{n-1}(A')+2,\quad k_{n-1}(A')\leq k_n(A). 
\end{equation}
Observe that we do not  bound the difference $k_n(A)-k_{n-1}(A')$ from above. Now we state a crucial estimate for the diameter of 
the set $A$:
\begin{equation}
\tag{P3} 
p_n(A)=0\ \Rightarrow\ \Delta(A,J_n(A),u_n(A),k_n(A),j_n(A))\leq 2^{n/2}r^{-j_n(A)},
\end{equation}
and its  version for a positive value of the counter $p_n(A)$:
\begin{equation}
\tag{P4}
p_n(A)>0\ \Rightarrow\ \Delta(A,J_n(A),u_n(A),k_n(A),j_n(A))\leq 2^{(n-p_n(A))/2}r^{-j_n(A)+1}.
\end{equation}
We require that ``parameters $k,J,u$ do not change unless $p_n(A)=1$"
\begin{equation}
\tag{P5}
p_n(A)\neq 1\ \Rightarrow\ u_n(A)=u_{n-1}(A'),\ k_n(A)=k_{n-1}(A'), J_n(A)=J_{n-1}(A').
\end{equation}
Next condition describes how parameters changes if $p_n(A)=1$:
\begin{align}
\notag 
p_n(A)= 1\ \Rightarrow\ &u_n(A)\in A',\ j_n(A)=j_{n-1}(A')+2  \mbox{ and }
\\
\tag{P6}
&J_n(A)=\{i\in J_{n-1}(A')\colon\ |u_{n}(A)_i-u_{n-1}(A')_i|\leq 2r^{-k_{n-1}(A')}\}.
\end{align}
For $p_n(A)>1$ parameter $j_n$ does not change 
\begin{equation}
\tag{P7}
p_n(A)>1\ \Rightarrow\ j_n(A)=j_{n-1}(A').
\end{equation}
Last two conditions describe the behavior of the counter $p_n$
\begin{equation}
\tag{P8}
p_n(A)>0\ \Rightarrow\ p_n(A)=p_{n-1}(A')+1,
\end{equation}
and
\begin{equation}
\tag{P9}
p_n(A)=0\Rightarrow\ p_{n-1}(A')\in\{0,2\kappa-1\},\ j_n(A)\leq j_{n-1}(A')+1.
\end{equation}

\begin{prop}
\label{prop:partconstr}
Suppose that $r=2^\kappa$, where  $\kappa$ is a sufficiently large positive integer and  
$T\subset \ell^2(I)$ satisfies $\Delta_2(T)\leq  r^{-j_0}$. Then 
there exists an admissible sequence of partitions  $(\cala_n)_{n\geq 0}$ of $T$, points $u_n(A)\in T$,
sets $J_n(A)\subset I$ and integers 
$k_n(A)\leq j_n(A)$, $0\leq p_n(A)\leq 2\kappa-1$, $A\in \cala_n$  which satisfy conditions (P1)-(P9).
Moreover for all $t\in T$,
\begin{equation}
\label{eq:estsum}
\sum_{n=0}^{\infty}2^{n}r^{-j_n(t)}\leq K(r)(r^{-j_0(T)}+b(T)),
\end{equation}
where $K(r)$ is a constant that depends only on $r$.
\end{prop}

\begin{proof}
Define $F_n(A):=F(A,J_n(A),u_n(A),k_n(A),j_n(A))$. We will additionally require the following two conditions, which
will help us to prove \eqref{eq:estsum}: first
\begin{equation}
\tag{P10}
p_n(A)=1\ \Rightarrow\ F_n(A)\leq F_{n-1}(A')-\frac{1}{\Lk}2^{n-1}r^{-j_n(A)+1},
\end{equation}
and second, \\
if $n\geq 2$, $p_n(A)=p_{n-1}(A')=0$ and $j_{n}(A)=j_{n-1}(A')$ then for any $D\subset A$ with  
$\Delta(D,J_n(A),u_n(A),k_n(A),j_n(A)+2)\leq \frac{1}{\Li}2^{(n-1)/2}r^{-j_n(A)-1}$ we have 
\\
\begin{align}
\notag
F(D,&J_n(A),u_n(A),k_n(A),j_n(A)+2)
\\
\tag{P11}
&\leq F(A,J_n(A),u_n(A),k_n(A),j_n(A)+2)-\frac{1}{\Lj}2^{n-1}r^{-j_n(A)-1}.
\end{align}

We assume that $\kappa$ is large enough so that $r\geq \max\{r_0,4\Li^2\}$, where $r_0$ is given by Corollary \ref{cor:maindec}.

We start the construction with ${\cal A}_0={\cal A}_1=\{T\}$, $k_1(T)=j_1(T)=k_0(T)=j_0(T)=j_0$,
$p_1(T)=p_0(T)=0$ and $u_1(T)=u_0(T)=t_0$, where $t_0$ is a point in $T$. Since
$\Delta(T,J_n(A),u_n(A),k_n(A),j_n(A))\leq \Delta_2(T)\leq r^{-j_0}$ conditions (P1)-(P11) are 
satisfied for $n\leq 1$.

Assume now that ${\cal A}_n$, $n\geq 1$ is already constructed and fix set $B\in {\cal A}_n$. We will split this set
into at most $N_n$ sets in ${\cal A}_{n+1}$ this way $|{\cal A}_{n+1}|\leq N_n|{\cal A}_n|\leq N_n^2=N_{n+1}$
as required.

If $1\leq p_n(B)\leq 2\kappa-2$ we do not split $B$. That is, we decide that $B\in {\cal A}_{n+1}$ and we set
$p_{n+1}(B):=p_n(B)+1$, $k_{n+1}(B):=k_n(B)$, $j_{n+1}(B):=j_{n}(B)$, $J_{n+1}(B):=J_n(B)$ and $u_{n+1}(B):=u_n(B)$. 
It is easy to see that all required conditions holds for $B$ and $n+1$. 

If $p_n(B)= 2\kappa-1$ we do not split $B$ either, but this time we set
$p_{n+1}(B):=0$, $k_{n+1}(B):=k_n(B)$, $j_{n+1}(B):=j_n(B)$, $J_{n+1}(B):=J_{n}(B)$ and $u_{n+1}(B):=u_n(B)$. 
The condition (P3) for $A=B$ and $n+1$ follows by (P4) for $A=B$.

Finally assume that $p_n(B)=0$ then we will split $B$ using Corollary \ref{cor:maindec} with $T=B$, $u=u_n(B)$,
$u'$ any point in $B$, $J=J_n(B)$,
$k=k_n(B)$ and $j=j_n(B)$. We obtain a partition $B=\bigcup_{l\leq m}A_l$, $m\leq N_{n}$ and each of the sets $A_l$ satisfies
one of the conditions (C1)-(C3). Let $A=A_l$ be one of these sets.

If $A$ satisfies (C1) we set
$p_{n+1}(A):=0$, $j_{n+1}(A):=j_n(B)$, $k_{n+1}(A):=k_n(B)$, $J_{n+1}(A):=J_{n}(B)$ and $u_{n+1}(A):=u_n(B)$.
Property (P11) for $A$ and $n+1$ follows now by (C1).

If $A$ satisfies (C2a)-(C2b) we define
$p_{n+1}(A):=1$, $j_{n+1}(A):=k_{n+1}(A)=j_n(B)+2$,  $u_{n+1}(A):=u'$
and
\[
J_{n+1}(A):=J'=\{i\in J_n(B)\colon\ |u_n(B)_i-u'_i|\leq 2r^{-k_n(A)}\}.
\]
Property (P4) for $A$ and $n+1$ follows by (C2a) and property (P10) by (C2b).

Finally if $A$ satisfies (C3) we define
$p_{n+1}(A):=0$, $j_{n+1}(A)=j_n(B)+1$, $k_{n+1}(A)=k_n(B)$, $J_{n+1}(A):=J_{n}(B)$ and $u_{n+1}(A)=u_n(B)$.
Condition (P3) for $A$ and $n+1$ now follows by (C3).

This way we constructed an admissible partition that satisfies (P1)-(P11). To finish the proof we need 
to show \eqref{eq:estsum}.

Observe that $F_n(A)\leq F_{n-1}(A')$: for $p_n(A)=1$ this obviously follows 
from  (P10), while for
$p_n(A)\neq 1$, we have $u_{n-1}(A')=u_n(A)$, $J_{n-1}(A')=J_n(A)$, $j_{n-1}(A')\leq j_n(A)$ and $k_{n-1}(A')=k_n(A)$
and we may use Proposition \ref{prop:monfunct}.

Fix $t\in T$ and define
$a_n=a_n(t):=2^nr^{-j_n(t)}$. If $p_{n}(t)=0$ and $n\geq 2$ then either $j_{n-1}(t)<j_n(t)$ and  $a_{n-1}>a_n$
or $j_{n-1}(t)=j_n(t)$, $p_{n-1}(t)=0$, which by (P11) gives $a_n\leq 2\Lj rF_n(t)\leq 2\Lj rb(T)$
or $p_{n-1}(t)=2\kappa-1$, which yields $p_{n-2\kappa}(t)=0$, $j_{n-2\kappa}(t)=j_n(t)-2$ and $a_{n-2\kappa}=a_n$.
If $p_n(t)>0$ then taking $n':=\inf\{m\geq n\colon\ p_{m}(t)=0\}$ we get $j_{n'}(t)=j_n(t)$, $p_{n'}(t)=0$
and $a_n<a_{n'}$. This shows that
\[
\sup_{n}a_n\leq \max\{a_0,a_1,2\Lj rb(T)\}<\infty.
\]

Let
\[
V_0:=\{n\geq 0\colon\  a_m<2^{|m-n|}a_n \mbox{ for all }m\geq 0,\ m\neq n\}.
\]
If $n\in V_0$ then $a_{n+1}=2^{n+1}r^{-j_{n+1}(t)}<2a_n=2^{n+1}r^{-j_n(t)}$, so that
\[
V_0\subset V_1:=\{n\geq 0\colon j_n(t)<j_{n+1}(t)\}.
\] 
By Lemma \ref{lem:smart} with $\alpha=2$ we have 
\[
\sum_{n\geq 0}a_n\leq 4\sum_{n\in V_0}a_n\leq 4\sum_{n\in V_1}a_n.
\]
Let us enumerate the elements of $V_1$ as $1\leq n_0<n_1<n_2<\ldots$ and set 
\[
V_2:=\{n_q\colon\ a_{n_{m}}< 2^{|m-q|}a_{n_q}\mbox{ for all }m\geq 0,\ m\neq q\}.
\]
Lemma \ref{lem:smart} applied once again implies 
\[
\sum_{n\geq 0}a_n\leq 4\sum_{n\in V_1}a_n\leq 16\sum_{n\in V_2}a_n.
\]

Fix $n=n_q\in V_2$. If $j_{n-1}(t)<j_n(t)$ then $n-1=n_{q-1}$ and (since $r\geq 4$)
\[
a_{n_{q-1}}=a_{n-1}\geq \frac{r}{2}a_n\geq 2a_n,
\]
which contradicts the definition of $V_2$. Hence $j_{n-1}(t)=j_n(t)<j_{n+1}(t)$. 
We have the following 4 possibilities.

1. $j_{n+1}(t)=j_n(t)+2$, then $p_{n+1}(t)=1$ and by (P10)
\[
a_n=r2^nr^{-j_{n+1}(t)+1}\leq \Lk r(F_{n}(t)-F_{n+1}(t)).
\]

2. $j_{n+1}(t)=j_n(t)+1$ and $j_{n_{q+1}+1}(t)=j_{n_{q+1}}(t)+2$ then $p_{n_{q+1}+1}(t)=1$ and  by (P10)
\[
a_n\leq \frac{1}{4}r^3a_{n_{q+1}+1}\leq \frac{1}{2}\Lk r^2(F_{n_{q+1}}(t)-F_{n_{q+1}+1}(t)).
\]

3. $p_{n-1}(t)=2\kappa-1$ then $p_{n-2\kappa+1}(t)=1$, $j_{n-2\kappa}(t)<j_{n-2\kappa+1}(t)=j_n(t)$,
so $n-2\kappa=n_{q-1}$ and by (P10) 
\begin{align*}
a_n=2^{2\kappa-1}a_{n_{q-1}+1}
&\leq 2^{2\kappa}\Lk r^{-1}(F_{n_{q-1}}(t)-F_{n_{q-1}+1}(t))
\\
&=\Lk r(F_{n_{q-1}}(t)-F_{n_{q-1}+1}(t)).
\end{align*}

4. $p_{n-1}(t)=0$, $j_{n+1}(t)=j_n(t)+1$ and $j_{n_{q+1}+1}(t)=j_{n_{q+1}}(t)+1$. 
Then $p_{n_{q+1}+1}(t)=0$, moreover by the definition of $V_2$
\[
2^{n_{q+1}}r^{-j_n(t)-1}=a_{n_{q+1}}<2a_{n_q}=2^{n+1}r^{-j_n(t)}
\]
which yields $n_{q+1}-n\leq \kappa$. In particular this implies $p_m(t)=0$ for all
$n\leq m\leq n_{q+1}+1$. Hence $k_{n_{q+1}+1}(t)=k_{n}(t)$, $j_{n_{q+1}+1}(t)=j_n(t)+2$, $u_{n_{q+1}+1}(t)=u_n(t)$
and $J_{n_{q+1}+1}(t)=J_n(t)$. Therefore (P3) used for $n=n_{q+1}+1$ and $A=A_{n_{q+1}+1}$ implies
\begin{align*}
\Delta(A_{n_{q+1}+1}(t),J_n(t),u_n(t),k_n(t),j_n(t)+2)
&\leq 2^{(n_{q+1}+1)/2}r^{-j_{n}(t)-2}
\\
&\leq \frac{1}{\Li}2^{(n-1)/2}r^{-j_n(t)-1},
\end{align*}
where the last estimate follows since $n_{q+1}-n\leq \kappa$ and $r=2^\kappa\geq (2\Li)^2$.
Then either $n=1$ or we may apply (P11)  to $D=A_{n_{q+1}+1}$ and get
\[
a_n\leq 2\Lj r(F_{n}(t)-F_{n_{q+1}+1}(t)).
\]

This shows that for $n=n_q\in V_2$,  either $n=1$ or $a_n\leq K(r)(F_{n_{l}}(t)-F_{n_{l+2}}(t))$ for some
$l\in\{q-1,q,q+1\}$. By  monotonicity of the map $l\mapsto F_{n_l}(t)$ this gives (with a value of $K(r)$ which may change
at each occurrence)
\[
\sum_{n\geq 0}a_n\leq 16\sum_{n\in V_2}a_n\leq 16a_1+K(r)F_0(T)\leq K(r)(r^{-j_0}+b(T)).
\]

\end{proof}

\section{Proofs of the Main Result}
\label{sec:proofs}

We are now ready to present proofs of the main Theorem \ref{th:BC} and Corollary \ref{cor:conv}.

\begin{proof}[Proof of Theorem \ref{th:BC}]
By homogeneity we may assume that $b(T)=1/4$ and then $\Delta_2(T)\leq 1$ by Lemma \ref{lem:diamest}.
We apply Proposition \ref{prop:partconstr} with $j_0=0$ and get an admissible sequence of partitions
$({\cal A}_n)_{n\geq 0}$, numbers $p_n(A),k_n(A),j_n(A)$ and points $u_n(A)$. First we inductively
define points $\pi_n(A)$. We set $\pi_0(T)=u_0(T)$ and for $A\in{\cal A}_n$, $n\geq 1$ we define
$\pi_n(A)=\pi_{n-1}(A')$ if $j_n(A)=j_{n-1}(A')$, $\pi_n(A)=u_n(A)$ if $p_n(A)=1$ and choose for
$\pi_n(A)$ an arbitrary point in $A$ if $p_n(A)=0$ and $j_n(A)>j_{n-1}(A')$.

As in Theorem \ref{th:part} we set 
\[
I_n(t):=\big\{i\in I\colon\ |\pi_{q+1}(t)_i-\pi_q(t)_i|\leq r^{-j_q(t)} \mbox{ for } 0\leq q\leq n-1\big\}.
\]

First we  show that
\begin{equation}
\label{eq:diff_pi_u}
|\pi_{n+1}(t)_i-u_n(t)_i|\leq 2r^{-k_n(t)}\quad \mbox{ for }i\in I_{n+1}(t).
\end{equation}
To this aim we define $J'=\{0\}\cup\{n\geq 1\colon\ p_n(t)=1\}$. Then $\pi_n(t)=u_n(t)$ for $n\in J'$.
Fix $n$ and let $n'$ be the largest element of $J'$ such that $n'\leq n$. Then by (P5) $u_n(t)=u_{n'}(t)=\pi_{n'}(t)$
and $k_n(t)=k_{n'}(t)$. Therefore for $i\in I_{n+1}(t)$,
\begin{align*}
|\pi_{n+1}(t)_i-u_n(t)_i|&=|\pi_{n+1}(t)_i-\pi_{n'}(t)_i|\leq \sum_{q=n'}^n|\pi_{q+1}(t)_i-\pi_{q}(t)_i|
\\
&\leq \sum_{j\geq j_{n'}(t)}r^{-j}
\leq 2r^{-j_{n'}(t)}\leq 2r^{-k_{n'}(t)}=2r^{-k_n(t)}.
\end{align*}

Now we inductively show that $I_n(t)\subset J_n(t)$. For $n=0$ both sets equals $I$. If $p_{n+1}(t)\neq 1$
then $I_{n+1}(t)\subset I_n(t)\subset J_{n}(t)=J_{n+1}(t)$ and if $p_{n+1}(t)=1$ then 
$\pi_{n+1}(t)=u_{n+1}(t)$ so by \eqref{eq:diff_pi_u}, $|u_{n+1}(t)-u_n(t)|\leq 2r^{-k_n(t)}$ for
$i\in I_{n+1}(t)$ hence by (P6) and the induction assumption $I_{n+1}(t)\subset J_{n+1}(t)$.

Finally assume
that $A\in {\cal A}_n$, $j_n(A)>j_{n-1}(A')$ and $t\in A$. Then $p_{n-1}(A')=0$, $t,\pi_{n}(A)\in A'$, 
$I_n(A)\subset J_n(A)\subset J_{n-1}(A')$ and $|\pi_n(A)_i-u_{n-1}(A')_i|\leq 2r^{-k_{n-1}(A')}$ for
$i\in I_n(A)$. Hence Lemma \ref{lem:estd} (applied with $J=I_n(A)$, $u=u_{n-1}(A')$, $s=\pi_n(A)$, 
$j=j{n-1}(A')$ and $k=k_{n-1}(A')$) and (P3) yield
\begin{align*}
\sum_{i\in I_n(A)}\min\{&(t_i-\pi_n(A)_i)^2,r^{-2j_n(A)}\}
\leq \sum_{i\in I_n(A)}\min\{(t_i-\pi_n(A)_i)^2,r^{-2j_{n-1}(A')}\}
\\
&\leq
2\Delta(A',J_{n-1}(A'),u_{n-1}(A'),j_{n-1}(A'),k_{n-1}(A'))^2
\\
&\leq 2^{n}r^{-2j_{n-1}(A')}\leq r^42^{n}r^{-2j_n(A)}.
\end{align*}

Therefore all assumptions of Theorem \ref{th:part} are satisfied with $M=r^4$ and Theorem \ref{th:BC}
follows by \eqref{eq:estdec} and \eqref{eq:estsum} (since $r^{-j_0}=1=4b(T)$).

\end{proof}

\begin{proof}[Proof of Theorem \ref{cor:conv}]
By Theorem \ref{th:BC} we know that $T\subset T_1+T_2$ with $\sup_{t\in T_1}\|t\|_1\leq Lb(T)$ and $g(T_2)\leq Lb(T)$.
Then
\[
T-T\subset (T_1-T_1)+(T_2-T_2)\subset \mathrm{conv}\{2(T_1-T_1),2(T_2-T_2)\}.
\]
Obviously $T_1-T_1\subset L\overline{\mathrm{conv}}\{e_i\colon\ i\in I\}$, where $(e_i)_{i\in I}$ is the canonical
basis of $\ell^2(I)$.
The majorizing measure theorem for Gaussian processes implies (cf. \cite[Theorem 2.1.8]{Tab1}) that we can find vectors
$(s^n)_{n\geq 1}$ in $\ell^2$ such that
$T_2-T_2\subset \overline{\mathrm{conv}}\{s^n\colon\ n\geq 1\}$ and 
$\sqrt{\log(n+1)}\|s_n\|_{2}\leq Lg(T_2)\leq Lb(T)$. To finish the proof it is enough to notice that
$\|X_{e_i}\|_p=\|\ve_i\|_p= 1$ for any $p>0$ and that by Khinthine's inequality
$\|X_t\|_p\leq L\sqrt{p}\|t\|_2$ for $p\geq 1$.
\end{proof}

\section{Selected Applications}
\label{sec:appl}

%%Fernique

The Bernoulli Conjecture was motivated by the following question of X. Fernique concerning  random Fourier series.
Let $G$ be a compact Abelian group and $(F,\|\ \|)$ be a complex Banach space. Consider (finitely many) vectors
$v_i\in F$ and characters $\chi_i$ on $G$. X.\ Fernique \cite{Fe2} showed that
\[
\Ex\sup_{h\in G}\Big\|\sum_{i}v_ig_i\chi_i(h)\Big\|
\leq L\Big(\Ex\Big\|\sum_{i}v_ig_i\Big\|+\sup_{\|x^*\|\leq 1}\Ex\sup_{h\in G}\Big|\sum_{i}x^*(v_i)g_i\chi_i(h)\Big|\Big)
\]
and asked whether similar bound holds if one replaces Gaussian r.v's by random signs. Theorem \ref{th:BC} yields
an affirmative answer.

\begin{thm}
\label{th:Fe}
For any compact Abelian group $G$ any finite collection of vectors $v_i$ in a complex Banach space $(F,\|\ \|)$
and characters $\chi_i$ on $G$ we have
\[
\Ex\sup_{h\in G}\Big\|\sum_{i}v_i\ve_i\chi_i(h)\Big\|
\leq L\Big(\Ex\Big\|\sum_{i}v_i\ve_i\Big\|+\sup_{\|x^*\|\leq 1}\Ex\sup_{h\in G}\Big|\sum_{i}x^*(v_i)\ve_i\chi_i(h)\Big|\Big).
\]
\end{thm}

\noindent
{\bf Remark.} Since $\chi_i(e)=1$, where $e$ is the neutral element of $G$ we have
\[
\max\Big\{\Ex\Big\|\sum_{i}v_i\ve_i\Big\|,\sup_{\|x^*\|\leq 1}\Ex\sup_{h\in G}\Big|\sum_{i}x^*(v_i)\ve_i\chi_i(h)\Big|\Big\}
\leq \Ex\sup_{h\in G}\Big\|\sum_{i}v_i\ve_i\chi_i(h)\Big\|.
\] 
Therefore Theorem \ref{th:Fe} gives a two-sided bound on $\Ex\sup_{h\in G}\|\sum_{i}v_i\ve_i\chi_i(h)\|$.

\begin{proof}[Proof of Theorem \ref{th:Fe}] We need to show that for any bounded set $T\subset \mathbb{C}^n$, $n<\infty$,
\begin{equation}
\label{eq:Fer1}
\Ex\sup_{h\in G,t\in T}\Big|\sum_{i=1}^n t_i\ve_i\chi_i(h)\Big|\leq
L\Big(\Ex\sup_{t\in T}\Big|\sum_{i=1}^n t_i\ve_i\Big|+\sup_{t\in T}\Ex\sup_{h\in G}\Big|\sum_{i=1}^nt_i\ve_i\chi_i(h)\Big|\Big).
\end{equation}
Let $M:=\Ex\sup_{t\in T}|\sum_{i=1}^n t_i\ve_i|$.
Theorem \ref{th:BC} implies that we can find a decomposition $T\subset T_1+T_2$, with $\sup_{t^1\in T_1}\|t^1\|_1\leq LM$ and
\[
\Ex\sup_{t^2\in T_2}\Big|\sum_{i=1}^n t_i^2g_i\Big|\leq LM.
\] 
Obviously
\begin{align}
\notag
\Ex\sup_{h\in G,t\in T}&\Big|\sum_{i=1}^n t_i\ve_i\chi_i(h)\Big|
\\
\label{eq:Fer2}
&\leq
\Ex\sup_{h\in G,t^1\in T_1}\Big|\sum_{i=1}^n t_i^1\ve_i\chi_i(h)\Big|+
\Ex\sup_{h\in G,t^2\in T_2}\Big|\sum_{i=1}^n t_i^2\ve_i\chi_i(h)\Big|.
\end{align}
Since $|\sum_{i=1}^n t_i^1\ve_i\chi_i(h)|\leq \sum_{i=1}^n|t_i^1||\chi_i(h)|=\|t^1\|_1$ we get
\begin{equation}
\label{eq:Fer3}
\Ex\sup_{h\in G,t^1\in T_1}\Big|\sum_{i=1}^n t_i^1\ve_i\chi_i(h)\Big|\leq \sup_{t\in T^1}\|t^1\|_1\leq LM.
\end{equation}
Estimate \eqref{eq:gaussdom} and Fernique's theorem imply
\begin{align}
\notag
\Ex\sup_{h\in G,t^2\in T_2}\Big|\sum_{i=1}^n &t_i^2\ve_i\chi_i(h)\Big|
\leq \sqrt{\frac{\pi}{2}}\Ex\sup_{h\in G,t^2\in T_2}\Big|\sum_{i=1}^n t_i^2 g_i\chi_i(h)\Big|
\\
\label{eq:Fer4}
&\leq L\Big(\Ex\sup_{t^2\in T_2}\Big|\sum_{i=1}^n t_i^2g_i\Big|
+\sup_{t^2\in T_2}\Ex\sup_{h\in G}\Big|\sum_{i=1}^nt_i^2 g_i\chi_i(h)\Big|\Big).
\end{align}
The Marcus-Pisier estimate \cite{MarP} yields for any $t^2\in T_2$,
\begin{equation}
\label{eq:Fer5}
\Ex\sup_{h\in G}\Big|\sum_{i=1}^nt_i^2 g_i\chi_i(h)\Big|\leq 
L\Ex\sup_{h\in G}\Big|\sum_{i=1}^nt_i^2 \ve_i\chi_i(h)\Big|.
\end{equation}
Since we may assume that $T_2\subset T-T_1$ we get
\begin{align}
\notag
\sup_{t^2\in T_2}\Ex\sup_{h\in G}&\Big|\sum_{i=1}^nt_i^2 \ve_i\chi_i(h)\Big|
\\
\notag
&\leq \sup_{t\in T}\Ex\sup_{h\in G}\Big|\sum_{i=1}^nt_i \ve_i\chi_i(h)\Big|
+\sup_{t^1\in T_1}\Ex\sup_{h\in G}\Big|\sum_{i=1}^nt_i^1 \ve_i\chi_i(h)\Big|
\\
\label{eq:Fer6}
&\leq \sup_{t\in T}\Ex\sup_{h\in G}\Big|\sum_{i=1}^nt_i \ve_i\chi_i(h)\Big|+LM.
\end{align}
Estimate \eqref{eq:Fer1} follows by \eqref{eq:Fer2}-\eqref{eq:Fer6}.
\end{proof}

%%maximal inequality

Another consequence of Theorem \ref{th:BC} is a Levy-Ottaviani type maximal 
inequality for VC-classes (see \cite{La1} for details). Recall that a class ${\cal C}$ of subsets of $I$ is called a 
Vapnik-Chervonenkis class
(or in short a VC-class) of order at most $d$ if for any set $A\subset I$ of cardinality $d+1$ we have
$|\{C\cap A\colon\ C\in{\cal C}\}|<2^{d+1}$.

\begin{thm}
\label{th:maxVC}
Let $(X_i)_{i\in I}$ be independent random variables in a separable Banach space $(F,\| \cdot\|)$ such that 
$|\{i\colon\ X_i\neq 0\}|<\infty$ a.s. and ${\cal C}$ be a countable VC-class of subsets of $I$ of order $d$. 
Then
\[
\Pr\Big(\sup_{C\in {\cal C}}\Big\|\sum_{i\in C}X_i\Big\|\geq u\Big)
\leq 
K(d)\sup_{C\in {\cal C}\cup\{I\}}\Pr\Big(\Big\|\sum_{i\in C}X_i\Big\|\geq \frac{u}{K(d)}\Big)\quad \mbox{ for }u>0,
\] 
where $K(d)$ is a constant that depends only on $d$. Moreover if the variables $X_i$ are symmetric then
\[
\Pr\Big(\sup_{C\in {\cal C}}\Big\|\sum_{i\in C}X_i\Big\|\geq u\Big)
\leq 
K(d)\Pr\Big(\Big\|\sum_{i\in I}X_i\Big\|\geq \frac{u}{K(d)}\Big)\quad \mbox{ for }u>0.
\] 
\end{thm}

It is easy to see (taking $F=\er$, $X_i=\ve_i v$ for $i\in I_0$ and $X_i=0$ otherwise, where $I_0$ is a finite subset of $I$
and $v$ is any nonzero vector in $F$)
that being a VC-class is a necessary assumption even in the scalar case. 

Maximal inequalities of this type may be used to derive It\^o-Nisio type theorems reducing almost sure statements 
to statements in probability and as a consequence obtain various limit type theorems for VC-classes. As an example
of application we present a uniform Strong Law of Large Numbers.

\begin{cor}
Let $(X_i)_{i\geq 1}$ be independent symmetric r.v's with values in a separable Banach space $(F,\|\ \|)$ such
that $\frac{1}{a_n}\sum_{i=1}^n X_i \rightarrow 0$ a.s.. Then for any VC-class  ${\cal C}$ of subsets of $\en$ we have
\[
\lim_{n\rightarrow\infty}\frac{1}{a_n}\max_{C\in {\cal C}}\Big\|\sum_{i\in C\cap \{1,\ldots,n\}}X_i\Big\|=0 
\mbox{ a.s..}
\]
\end{cor}

\begin{proof}
Let $n_0$ be a fixed positive integer. Then for any $A\subset \en$
\[
\max_{n\geq n_0}\frac{1}{a_n}\Big\|\sum_{i\in A\cap\{1,\ldots,n\}}X_i\Big\|=\Big\|\sum_{i\in A}Y_i\Big\|,
\]
where $Y_i$ are random variables in $\ell^{\infty}(F)$ given by $Y_i(n)=0$ for $n<n_0$ or $i>n$ and 
$Y_i(n)=X_i$ for $i\leq n\geq n_0$. Applying Theorem \ref{th:maxVC} to random variables $Y_i$ we get for any $t>0$,
\[
\Pr\Big(\max_{n\geq n_0}\frac{1}{a_n}\max_{C\in {\cal C}}\Big\|\sum_{i\in C\cap \{1,\ldots,n\}}X_i\Big\|\geq t\Big)
\leq K\Pr\Big(\max_{n\geq n_0}\frac{1}{a_n}\Big\|\sum_{i=1}^nX_i\Big\|\geq \frac{t}{K}\Big)
\]
where $K$ is a constant that depends only on ${\cal C}$ and the assertion easily follows.
\end{proof}

%%\bigskip{\sc OK, give a simple concrete statement using the force of your result so that one can make sense out of it.}
%%\bigskip

\begin{proof}[Sketch of the proof of Theorem \ref{th:Fe}] It is rather a standard exercise to reduce to the case when
$I$ is finite and $X_i=v_i\ve_i$ for some vectors $v_i\in F$.   
Using concentration properties of Bernoulli processes it is enough to show that for any bounded symmetric set 
$T\subset \er^I$ and any VC-class of order $d$,
\begin{equation}
\label{eq:VC1}
\Ex\sup_{C\in {\cal C}}\sup_{t\in T}\Big|\sum_{i\in C}t_i\ve_i\Big|\leq 
K(d)\Ex\sup_{t\in T}\Big|\sum_{i\in I}t_i\ve_i\Big|=K(d)b(T).
\end{equation}

Let $T\subset T_1+T_2$ be a decomposition given by Theorem \ref{th:BC}. We may also assume that $T_1$ and $T_2$ are symmetric.
Obviously 
$|\sum_{i\in C}t_i^1\ve_i\Big|\leq \sum_{i\in C}|t_i^1|\leq \|t^1\|_1$, hence
\begin{equation}
\label{eq:VC2}
\Ex\sup_{C\in {\cal C}}\sup_{t^1\in T_1}\Big|\sum_{i\in C}t_i^1\ve_i\Big|\leq \sup_{t^1\in T_1}\|t^1\|_1\leq Lb(T). 
\end{equation}
Inequality \eqref{eq:gaussdom} gives
\begin{equation}
\label{eq:VC3}
\Ex\sup_{C\in {\cal C}}\sup_{t^2\in T_2}\Big|\sum_{i\in C}t_i^2\ve_i\Big|\leq
\sqrt{\frac{\pi}{2}}\Ex\sup_{C\in {\cal C}}\sup_{t^2\in T_2}\Big|\sum_{i\in C}t_i^2g_i\Big|.
\end{equation}
The result of Krawczyk \cite{Kr} and the choice of $T_2$ yields
\begin{equation}
\label{eq:VC4}
\Ex\sup_{C\in {\cal C}}\sup_{t^2\in T_2}\Big|\sum_{i\in C}t_i^2g_i\Big|\leq K(d)g(T_2)\leq K(d)b(T).
\end{equation}
Estimates \eqref{eq:VC2}-\eqref{eq:VC4} imply \eqref{eq:VC1}.
\end{proof}

\noindent
{\bf Remark.} Alternatively one may prove \eqref{eq:VC1} using Corollary \ref{cor:conv} and the fact that maximal
inequalities hold for $F=\er$.

\section{Further Questions}
\label{sec:probl}

It is natural to ask for  bounds on suprema for another classes of stochastic processes.
The majorizing measure upper bound 
works in quite general situations, cf. \cite{Be}. 
Two-sided estimates are known however only in very few cases. For
``canonical processes'' of the form $X_t=\sum_{i\geq 1}t_iX_i$, where $X_i$ are independent centered r.v's  results in the 
spirit of Corollary \ref{cor:conv} were obtained for certain symmetric variables with log-concave tails \cite{Ta_canon,La0}.

A basic important class of canonical processes worth investigation is a class of ``selector processes" of the form
\[
X_t=\sum_{i\geq 1}t_i(\delta_i-\delta),\quad t\in \ell^2,
\]
where $(\delta_i)_{i\geq 1}$ are independent random variables such that $\Pr(\delta_i=1)=\delta=1-\Pr(\delta_i=0)$.
We may bound the quantity
\[
\delta(T):=\Ex\sup_{t\in T}\Big|\sum_{i\geq 1}t_i(\delta_i-\delta)\Big|,\quad T\subset \ell^2
\]
in two ways. 

First bound for $\delta(T)$ follows by a pointwise estimate. Namely let $(\delta_i')_{i\geq 1}$ be an independent copy of 
$(\delta_i)_{i\geq 1}$, then by Jensen's inequality,  
\[
\delta(T)\leq \Ex\sup_{t\in T}\Big|\sum_{i\geq 1}t_i(\delta_i-\delta_i')\Big|\leq
2\Ex\sup_{t\in T}\Big|\sum_{i\geq 1}t_i\delta_i\Big|\leq 2\Ex\sup_{t\in T}\sum_{i\geq 1}|t_i|\delta_i.
\]

Second estimate is based on chaining. To introduce it we define for $\alpha>0$ and a metric space $(T,d)$,
\[
\gamma_{\alpha}(T,d):=\inf\sup_{t\in T}\sum_{n=0}^{\infty}2^{n/\alpha}\Delta(A_n(t)),
\]
where as in the definition of $\gamma_2$ the infimum runs over all admissible sequences of partitions
$({\cal A}_n)_{n\geq 0}$ of the set $T$.
Bernstein's inequality implies that for $X_t=\sum_{i\geq 1}t_i(\delta_i-\delta)$ and $\delta\in (0,1/2]$ we have
\[
\Pr(|X_t-X_s|\geq u)\leq 2\exp\Big(-\min\Big\{\frac{u^2}{L\delta d_2(s,t)^2},\frac{u}{Ld_\infty(s,t)}\Big\}\Big)
\quad \mbox{for }s,t\in \ell^2,
\]
where $d_p(t,s):=\|t-s\|_p$ denotes the $\ell^p$-distance.
This together with a chaining argument \cite[Theorem 1.2.7]{Tab1} yields 
\[
\delta(T)\leq L(\sqrt{\delta}\gamma_2(T,d_2)+\gamma_1(T,d_\infty)).
\]

The next conjecture, formulated by M.\ Talagrand \cite{Ta_ECM}, states that there are no other ways to bound
$\delta(T)$ as the combination of the above two estimates and the fact that $\delta(T_1+T_2)\leq \delta(T_1)+\delta(T_2)$.

\begin{conj}
\label{conj:delta}
Let $0<\delta\leq 1/2$, $\delta_i$ be independent random variables such that $\Pr(\delta_i=1)=\delta=1-\Pr(\delta_i=0)$
and $\delta(T):=\Ex\sup_{t\in T}|\sum_{i\geq 1}t_i(\delta_i-\delta)|$ for $T\subset \ell^2$. Then for any set $T$ with 
$\delta(T)<\infty$ one may find a decomposition $T\subset T_1+T_2$ such that
\[
\Ex\sup_{t\in T_1}\sum_{i\geq 1}|t_i|\delta_i\leq L\delta(T),\ \ 
\sqrt{\delta}\gamma_2(T_2,d_2)\leq L\delta(T),\ \ \mbox{ and }\ \  \gamma_1(T_2,d_\infty)\leq L\delta(T).
\]
\end{conj}

It may be showed that for $\delta=1/2$ the above conjecture follows from Theorem \ref{th:BC}. 

Since any mean zero random variable is a mixture of mean zero two-points random variables selector processes are
strictly related to empirical processes 
\[
Z_f:=\frac{1}{\sqrt{N}}\sum_{i\leq N}(f(X_i)-\Ex f(X_i)),\quad f\in {\cal F},
\]
where $(X_i)_{i\leq N}$ are i.i.d. random variables and ${\cal F}$ is a class of measurable functions. 
Let 
\[
S_N({\cal F}):=\Ex\sup_{f\in {\cal F}}|Z_f|=
\frac{1}{\sqrt{N}}\Ex\sup_{f\in {\cal F}}\Big|\sum_{i\leq N}(f(X_i)-\Ex f(X_i))\Big|.
\]
As for selector processes there are two distinct ways to bound $S_N({\cal F})$. The first one is to use the
trivial pointwise bound $|\sum_{i\leq N}f(X_i)|\leq \sum_{i\leq N}|f(X_i)|$. The second is based on chaining and 
Bernstein's inequality
\begin{equation}
\label{eq:Ber}
\Pr\Big(\Big|\sum_{i\leq N}(f(X_i)-\Ex f(X_i))\Big|\geq t\Big)
\leq 
2\exp\Big(-\min\Big\{\frac{t^2}{4N\|f\|_2^2},\frac{t}{4\|f\|_{\infty}}\Big\}\Big),
\end{equation}
where $\|f\|_p$ denotes the $L_p$ norm of $f(X_i)$. Similar chaining arguments as in the case of selector processes
give 
\[
S_N({\cal F})\leq L\Big(\gamma_2({\cal F}_2,d_2)+\frac{1}{\sqrt{N}}\gamma_1({\cal F}_2,d_\infty)\Big),
\] 
where $d_p(f,g):=\|f-g\|_p$.

The following conjecture asserts that there are no other ways to bound suprema of empirical processes.

\begin{conj}
\label{conj:emp}
Suppose that ${\cal F}$ is a countable class of measurable functions. Then one can find 
a decomposition ${\cal F}\subset {\cal F}_1+{\cal F}_2$ such that
\[
\Ex \sup_{f_1\in {\cal F}_1}\sum_{i\leq N}|f_1(X_i)|\leq \sqrt{N}S_N({\cal F}),
\]
\[
\gamma_2({\cal F}_2,d_2)\leq LS_N({\cal F}) \quad \mbox{ and }\quad  \gamma_1({\cal F}_2,d_\infty)\leq L\sqrt{N}S_N({\cal F}).
\]
\end{conj}

Related conjectures with a much more detailed discussion may be found in \cite{Ta_STOC} and \cite[Chapter 12]{Tab1}.

\noindent
Institute of Mathematics\\
University of Warsaw\\
Banacha 2\\
02-097 Warszawa\\
Poland\\
\texttt{wbednorz@mimuw.edu.pl, rlatala@mimuw.edu.pl}

\end{document}